\documentclass[hidelinks,onefignum,onetabnum]{siamart250211}
\usepackage[utf8]{inputenc}
\usepackage[T1]{fontenc}
\usepackage[margin=1.1in]{geometry}
\usepackage{amsmath,amssymb,mathabx}
\usepackage{hyperref}
\usepackage{cleveref}
\usepackage{algorithm,algorithmic}
\usepackage{color}
\usepackage{graphicx}
\usepackage[normalem]{ulem}
\usepackage{tikz}
\usepackage{subfigure}

\Crefname{subsection}{Section}{Sections}
\crefname{subsection}{section}{sections}

\Crefname{ALC@unique}{Line}{Lines}
\newcounter{myalg}
\AtBeginEnvironment{algorithmic}{\refstepcounter{myalg}}
\makeatletter
\@addtoreset{ALC@unique}{myalg}
\makeatother

\newcommand{\R}{\mathbb{R}}
\newcommand{\V}[1]{\boldsymbol{#1}}
\newcommand{\M}[2][]{#1{\boldsymbol{#2}}}
\newcommand{\ME}[3][]{#1{\lowercase{#2}}_{#3}}

\newcommand{\bmat}[1]{\begin{bmatrix} #1 \end{bmatrix}}
\newcommand{\T}[2][]{#1{\boldsymbol{\mathcal{#2}}}}
\newcommand{\TE}[3][]{#1{\lowercase{#2}}_{#3}}
\newcommand{\prob}{\mathbb{P}}
\newcommand{\expect}{\mathbb{E}}
\newcommand{\mc}[1]{\mathcal{#1}}
\renewcommand{\t}{^\top}
\newcommand{\hada}{\ast}
\newcommand{\kron}{\otimes}
\newcommand{\kr}{\odot}
\newcommand{\krt}{\kr\t}

\newcommand{\krp}{\text{KRP}}   

\usepackage{todonotes}
\newcommand{\changes}[1]{{#1}}

\title{Improved Analysis of Khatri-Rao Random Projections and Applications \thanks{Submitted to the editors DATE. \funding{AKS was funded, in parts, by the National Science Foundation through the award DMS-1821149 and the Department of Energy through the awards DE-SC0023188 and DE-SC0025262. GB and BDV were funded by the National Science Foundation through the award CCF-1942892 and the Department of Energy under awards DE-SC-0023296 and DE-SC-0025394.}}}
\author{Arvind K.\ Saibaba\thanks{Department of Mathematics, North Carolina State University, \email{asaibab@ncsu.edu}}, \url{https://asaibab.math.ncsu.edu/}  \and Bhisham Dev Verma\thanks{Department of Computer Science, Wake Forest University, \email{vermabd@wfu.edu}} \and Grey Ballard\thanks{Department of Computer Science, Wake Forest University, \email{ballard@wfu.edu}}, \url{https://users.wfu.edu/ballard/}}

\begin{document}
\maketitle

\begin{abstract}
    Randomization has emerged as a powerful set of tools for large-scale matrix and tensor decompositions. 
    Randomized algorithms involve computing sketches with random matrices. 
    A prevalent approach is to take the random matrix as a standard Gaussian random matrix, for which the theory is well developed. 
    However, this approach has the drawback that the cost of generating and multiplying by the random matrix can be prohibitively expensive. 
    Khatri-Rao random projections (\krp{}s), obtained by sketching with Khatri-Rao products of random matrices, offer a viable alternative and are much cheaper to generate. 
    {However, the theoretical guarantees of using \krp{}s are much more pessimistic compared to their accuracy observed in practice.
    We attempt to close this gap by obtaining improved analysis of the use of \krp{}s in matrix and tensor low-rank decompositions.}
    We propose and analyze a new algorithm for low-rank approximations of block-structured matrices (e.g., block Hankel) using \krp{}s. 
    We also show how to accelerate tensor computations in the Tucker format using \krp{}s and give theoretical guarantees of the resulting low-rank approximations.  
    Numerical experiments on synthetic and real-world tensors show the computational benefits of the proposed methods. 
\end{abstract}


\section{Introduction}
Applications of low-rank approximations of matrices and tensors abound in science and engineering. For matrices, the singular value decomposition (SVD) gives an optimal low-rank representation, but it is expensive to compute for large-scale matrices. 
While considering tensor decompositions, there are many formats to consider, such as canonical polyadic (CP), Tucker, Tensor Train, among others~\cite{AGHKT14,ballard2025tensor,Comon14,Hackbusch14,kolda2009tensor,Morup11,SL+17}. In this paper, we focus on the Tucker decomposition, which is known to have good compression ability. 

In recent years, randomized algorithms for computing these low-rank matrix and tensor decompositions have become increasingly powerful~\cite{halko2011finding,martinsson2020randomized,murray2023randomized}. At the heart of randomized algorithms is the use of a random matrix to obtain a ``sketch'' of the matrix or tensor. A major computational bottleneck in randomized algorithms is the computational cost of forming the sketch.  There are two factors one must consider. The first is the cost of generating the random matrix, and second, the cost of applying the matrix or tensor to the random matrix to obtain the sketch. The prototypical random matrix is the standard Gaussian random matrix (a matrix with independent and identically distributed normal random variables with zero mean and unit variance). Gaussian random matrices have good theoretical properties that enable their wide-spread use in randomized algorithms. However, Gaussian random matrices have two major downsides. First, they can be expensive to generate and store, especially when the size of the matrix/tensor is large. Second, one cannot exploit the structure of the matrix or tensor to compute the sketch efficiently. 

In recent years, there is an ongoing search for finding random matrices for use in randomized algorithms that have comparable theoretical properties as Gaussian random matrices but can be used to efficiently compute the sketch~\cite{bujanovic2024subspace,kressner2017recompression,minster2024parallel,saibaba2025randomized,SGTU21}. 
In this paper, we consider Khatri-Rao random projections (\krp{}s), where the sketching matrix is a Khatri-Rao product of $d$ random matrices. 
This paper proposes and analyzes randomized algorithms for computing low-rank approximations of matrices and tensors that make use of \krp{}s. 
The main advantages of the proposed methods are (1) the need to generate fewer (pseudo) random numbers and lower storage costs to store the random matrices used in the compression, and (2) the ability to exploit the structure of the matrix/tensor, leading to a reduction in computational costs.

\paragraph{Contributions} A brief history of \krp{}s is given in~\cite[Section 9.4]{martinsson2020randomized}. See also~\cite[Chapter 7]{murray2023randomized} for additional references. However, to the best of our knowledge, the theoretical properties of \krp{}s in the context of low-rank matrix and tensor decompositions are not fully understood. This article aims to develop a theoretical foundation for the use of \krp{}s in low-rank matrix and tensor decompositions. A second goal of this article is to develop new low-rank approximation algorithms that use \krp{}s. Specifically, the contributions of the paper are:
\begin{enumerate}
    \item We provide theoretical analysis for the use of \krp{}s in the randomized range finding algorithm (\Cref{sec:rrf}), building on the analysis of~\cite{haselby2023fast}. As an intermediate step, we derive an oblivious subspace embedding result for \krp{}s. 
    
    \item In \Cref{sec:blockstruct}, we show how \krp{}s can be used to compute low-rank approximations of block-structured matrices (e.g., block Toeplitz, block Hankel, etc). We illustrate this approach on a test problem in system identification that requires compressing a large block Hankel matrix. An extension to multilevel block-structured matrices is discussed, and an analysis of the low-rank approximation is given. 
    \item We show \krp{}s can be used to accelerate the computational cost of randomized tensor decompositions in the HOSVD and  STHOSVD style and analyze the accuracy of the algorithms (\Cref{sec:tucker}). The memoized variant of HOSVD with \krp{}s  have previously not appeared in the literature.  
    \item We consider a recent application of tensor compression to optimal sensor placement in flow reconstructions. We show how \krp{}s can be used to accelerate computational costs on a challenging 3D unsteady flow problem (\Cref{sec:flow}). 
    \item We give a theoretical justification for the approach in~\cite{kressner2017recompression} that uses \krp{}s to recompress Hadamard products of tensors in the Tucker format (\Cref{sec:tucker}).
\end{enumerate}

The source code for the reproducibility of our experimental results is available at: \url{https://github.com/bhisham123/KRP-Random-Projections}. 

\paragraph{Related work}
A review of the various options for random matrices inside sketching are given in~\cite{martinsson2020randomized,murray2023randomized}. Theoretical guarantees for some other distributions may be found in~\cite{saibaba2025randomized}. 

We now turn our attention to \krp{}s. First, we address the special case of $d=2$. 
Sun et al.~\cite{SGTU21} establish that random matrices with \krp{} structure satisfy Johnson-Lindenstrauss type results for $d=2$. 
Similar results can be found in~\cite{chen2021tensor}, with applications to constrained least squares problems. The paper~\cite{bujanovic2024subspace} further used the Johnson-Lindenstrauss type result to establish a subspace embedding for random matrices with \krp{} structure, with applications to eigensolvers using contour integration. Bujanovic and Kressner~\cite{bujanovic2021norm} use random vectors with \krp{} structure to estimate norms and trace.

Rakhshan and Rabusseau~\cite{RR20} and Haselby et al.~\cite{haselby2023fast} further extend the Johnson-Lindenstrauss type results to random matrices with \krp{} structure for dimensions $d > 2$. 
These results were then extended to show subspace embeddings as in~\cite{bujanovic2024subspace} for $d=2$. Another more recent paper~\cite{camano2025faster} focuses on the analysis of \krp{}s using subspace injections. { 
A more detailed comparison of our results with these existing approaches is given in \cref{ssec:rrfproof}. } 
More extensive discussion on random matrices with \krp{} structure can be found in~\cite[Section 9.4]{martinsson2020randomized} and~\cite[Chapter 7]{murray2023randomized}.

\krp{}s have previously been used for tensor compression~\cite{CW19,che2025efficient,kressner2017recompression,sun2020low}, with theoretical guarantees for \krp{}s given in~\cite{CW19,che2025efficient,kressner2017recompression}. A comparison of our results with~\cite{CW19,che2025efficient} is given in \Cref{sec:rrf}. In~\cite{kressner2017recompression}, the authors use KRPs to recompress a Hadamard product of tensors in the Tucker format.
A detailed review of randomized techniques in tensor decompositions is given in~\cite[Section V]{randtucker_survey}. 
Some other recent methods include~\cite{bucci2024multilinear,hashemi2023rtsms}.

A different style of tensor random products is used in~\cite{bamberger2021hanson,jin2021faster,malik2020guarantees} with the form $\M\Phi\M{D}$, where $\M{D}$ is a diagonal matrix with \krp{} structure on the diagonal and $\M\Phi$ is a subsampled Fourier/Hadamard transform. Other papers using \krp{} structure include~\cite{iwen2021lower,LK22}. 
In~\cite{minster2021efficient,saibaba2021efficient}, random matrices using Kronecker products were used instead of \krp{}s.

\paragraph{Hardware and software description}
The experimental results reported in~\Cref{sec:blockstruct,sec:flow} are collected on a Mac Pro with an M1 chip and 16 GB of RAM, running MATLAB version R2023B. 
The results in~\Cref{sec:tucker} are collected on an Intel Xeon Gold 6226R system with 256 GB of RAM, running Ubuntu 20.04.6 LTS and MATLAB version R2023B.

\section{Background}
\subsection{Notation}
We denote matrices using uppercase bold italic letters, vectors using lowercase bold italic letters, and scalars using plain lowercase italic letters. For example,  $\M{X} \in \mathbb{R}^{m \times n}$ represents a matrix of size $m \times n$, $\V{x} \in \R^n$ denotes an $n$-dimensional vector, and $x$ represents a scalar. 

We use $\M{X} = \M{U} \M{\Sigma} \M{V}\t$ to denote the SVD of a matrix $\M{X}$, where $\M{U}$ and $\M{V}$ are orthogonal matrices containing the left and right singular vectors, respectively, and $\M{\Sigma}$ is a diagonal matrix whose diagonals contain the singular values of $\M{X}$. We use $\M{X}^\dagger$ to denote the Moore-Penrose pseudo-inverse of a matrix $\M{X}$.

We use $\kron$ to denote the Kronecker product and $\kr$ for the Khatri-Rao product. For matrices $\M{X} \in \R^{m \times n}$ and $\M{Y} \in \R^{p \times q}$, the \textbf{Kronecker product} $\M{X} \kron \M{Y} \in \R^{mp \times nq}$ is defined as the block matrix 
\[ \M{X} \kron \M{Y} = \bmat{x_{11} \M{Y} & \dots & x_{1n}\M{Y} \\ \vdots & \ddots & \vdots \\ x_{m1}\M{Y} & \dots & x_{mn}\M{Y}  }. \] 
The \textbf{Khatri-Rao product}, or the column-wise Kronecker product,  for matrices $\M{X} \in \R^{m \times r}$ and $\M{Y} \in \R^{n \times r}$, is defined as 
\[ \M{X} \odot \M{Y} = \bmat{ \V{x}_1 \kron \V{y}_1 & \dots &\V{x}_r \kron \V{y}_r  } \in \R^{mn \times r} , \] 
where the $i$-th column of the result is given by $\V{x}_i \kron \V{y}_i$, with $\V{x}_i$ and $\V{y}_i$ being the $i$-th columns of $\M{X}$ and $\M{Y}$, respectively.

We denote tensors by uppercase bold calligraphic letters, e.g., $\T{X} \in \mathbb{R}^{n_1 \times \cdots \times n_d}$ denotes an order-$d$ tensor, with entries $x_{i_1,\dots,i_d}$ for $1 \le i_j \le n_j$ and $1 \le j \le d$. A tensor can be unfolded along any of its modes. The mode-$i$ unfolding of a tensor $\T{X}$, denoted by $\M{X}_{(i)}$, is a matrix of size $n_i \times \prod_{j \neq i} n_j$, where each column corresponds to a mode-$i$ fiber of the tensor. 

Let $1 \le i \le d$. For a matrix $\M{A} \in \R^{r_i \times n_i}$, the tensor-times-matrix (TTM) product $\T{X} \times_i \M{A}$ contracts the tensor along mode-$i$, resulting in a new tensor $\T{Y} \in \R^{n_1 \times \cdots \times n_{i-1} \times r_i \times n_{i+1} \times \cdots \times n_d}$. This operation can also be expressed using the mode-$i$ unfolding as $\M{Y}_{(i)} = \M{A} \M{X}_{(i)}$. This operation extends naturally to $d$ matrices $\M{A}_i \in \mathbb{R}^{r_i \times n_i}$ applied along each mode and results in a multilinear transformation given by $\T{Y} = \T{X} \times_1 \M{A}_1 \times_2 \cdots \times_d \M{A}_d$. In terms of mode-$i$ unfolding, we can write it as $\M{Y}_{(i)} = \M{A}_i \M{X}_{(i)} (\M{A}_d \kron \cdots \kron \M{A}_{i+1} \kron \M{A}_{i-1} \kron \cdots \kron \M{A}_1)$. 

We denote the norm of a tensor $\T{X}$ by $\|\T{X}\|_F = \left(\sum_{i_1,\dots,i_d}|x_{i_1,\dots,i_d}|^2 \right)^{1/2}$, which is the square root of the sum of squares of all its entries.

\subsection{Randomized low-rank approximations} 
In recent years, randomized algorithms have gained significant popularity for computing low-rank approximations of matrices due to their reduced computational cost. {Some notable examples are the randomized range finder~\cite{halko2011finding}, randomized SVD~\cite{halko2011finding}, and single-\changes{pass}  randomized algorithms~\cite{nakatsukasa2020fast,tropp2017practical}.}
 Given a matrix $\M{X} \in \R^{m \times n}$, a target rank $r$ and an oversampling parameter $\rho$, we construct a sketch matrix by multiplying $\M{X}$ with a random matrix $\M{\Omega} \in \R^{n \times \ell }$ with $\ell = r + \rho$, that is, $\M{Y} = \M{X} \M{\Omega}$.  Typically, the entries of $\M{\Omega}$ are independent and drawn from a standard Gaussian distribution (with zero mean and unit variance), although other distributions can also be used. In particular, we explore a random matrix based on tensor random projections.

After computing the sketch $\M{Y} = \M{X\Omega}$, we compute a thin QR decomposition of $\M{Y} = \M{QR}$. If the matrix $\M{X}$ is approximately of rank $r$, or if its singular values decay rapidly after the $r$-th one, then the range of $\M{Q}$ serves as a good approximation of the range of $\M{X}$. This procedure is called the Randomized Range Finder, and its pseudocode is given in \cref{alg:rrf}.

\begin{algorithm}[!ht]
\caption{Randomized Range Finder~\cite{halko2011finding}}
\begin{algorithmic}[1]
\REQUIRE Matrix $\M{X} \in \R^{m\times n}$, target rank $r$ and oversampling parameter $\rho$.
\STATE Draw matrix  $\M{\Omega} \in \R^{n \times \ell}$ where $\ell = r+\rho$
\STATE Compute sketch matrix $\M{Y} = \M{X} \M{\Omega}$
\STATE Compute thin QR factorization $\M{Y} = \M{QR}$
\RETURN Matrix $\M{Q}$
\end{algorithmic}
\label{alg:rrf}
\end{algorithm}
The projection $\M{QQ}\t \M{X}$ provides an effective low-rank approximation of $\M{X}$; however, it yields an approximation of rank $\ell$, not the target rank $r$. To obtain a rank-$r$ approximation, a further truncation step is required. One common approach is to compute a singular value decomposition (SVD) of $\M{Q}\t\M{X}$, and retain only the leading $r$ components. This procedure is known as the Randomized SVD algorithm, and its pseudocode is presented in \cref{alg:rsvd}.
 
 \begin{algorithm}[!ht]
\caption{Randomized SVD~\cite{halko2011finding}}
\begin{algorithmic}[1]
\REQUIRE Matrix $\M{X} \in \R^{m\times n}$, target rank $r$ and oversampling parameter $\rho$.
\STATE Invoke \cref{alg:rrf} with inputs $\M{X} \in \R^{m\times n}$, target rank $r$ and oversampling parameter $\rho$ to obtain matrix $\M{Q}$ with orthonormal columns
\STATE Compute thin SVD $\M{Q}^T\M{X} = \hat{\M{U}} \M{\Sigma} \M{V}^T$ \label{line:secondpass}
\STATE Truncate $\M{U} = \M{Q} \hat{\M{U}}(:,1:r)$, $\M{\Sigma} = \M{\Sigma}(1:r,1:r)$ and $\M{V} = \M{V}(:,1:r)$
\RETURN Matrices $\M{U}$, $\M{\Sigma}$, $\M{V}$
\end{algorithmic}
\label{alg:rsvd}
\end{algorithm}

Several studies have proposed single-\changes{pass} randomized algorithms~\cite{nakatsukasa2020fast,tropp2017practical} for computing low-rank approximations of rectangular matrices. In this work, we focus on the single-\changes{pass} algorithm introduced by Tropp et al.~\cite {tropp2017practical}. Their method uses a two-sided random projection approach, where different random matrices are applied on the left and right to form sketch matrices. These sketches are then used to construct a low-rank approximation of the input matrix.
We discuss the process as follows:  Given a matrix  $\M{X} \in \mathbb{R}^{m \times n}$ and a target rank  $r$, compute the sketch matrices
\[
\M{Y} = \M{X} \M{\Omega} \quad \text{and} \quad \M{Z} = \M{\Psi}^\top \M{X},
\]
where $ \mathbf{\Omega} \in \mathbb{R}^{n \times \ell_r}$ and $\mathbf{\Psi} \in \mathbb{R}^{m \times \ell_l}$ are random Gaussian matrices with $ r < \ell_r < \ell_l< \min\{m,n\}$.  
Then, compute the QR decomposition of $\M{Y} = \M{Q} \M{R}$. The low rank approximation of $\M{X}$ is then computed as follows:
$$
    \M{X} \approx \M{Q} (\M{\Psi}\t \M{Q})^{\dagger} \M{Z} = \M{Q} \M{W}
$$
where $\M{W} = (\M{\Psi}\t \M{Q})^{\dagger} \M{Z}$.
A detailed pseudocode is provided in~\cref{alg:gen_nystrom}. This procedure yields an approximation of rank greater than $r$. We can obtain the exact rank-$r$ approximation through additional post-processing of $\M{W}$.

 \begin{algorithm}[!ht]
\caption{Single \changes{pass} randomized algorithm~\cite{tropp2017practical}}
\begin{algorithmic}[1]
\REQUIRE Matrix $\M{X} \in \R^{m\times n}$, sketching parameters $\ell_r$ and $\ell_l$
\STATE Draw random matrices  $\M{\Omega} \in \R^{n \times \ell_r}$ and $\M{\Psi} \in \R^{m \times \ell_l}$
\STATE Compute the sketches $\M{Y} = \M{X} \M{\Omega}$ and $\M{Z} = \M{\Psi}^\top \M{X}$
\STATE Compute QR factorization $\M{Y} = \M{Q} \M{R}$
\STATE Solve a least squares problem to obtain $\M{W} = (\M{\Psi}^T \M{Q})^{\dagger} \M{Z}$
\RETURN Matrices $\M{Q}$, $\M{W}$
\end{algorithmic}
\label{alg:gen_nystrom}
\end{algorithm}

\subsection{Tucker decomposition}
The Tucker decomposition is a higher-order generalization of matrix SVD to tensors. Let $\T{X} \in \R^{n_1 \times \cdots  \times n_d}$ be an order-$d$ tensor.  The Tucker decomposition with multilinear rank $(r_1, \ldots, r_d)$, where each $r_i = \operatorname{rank}(\M{X}_{(i)})$, factorizes $\T{X}$ into a core tensor $\T{G} \in \R^{r_1 \times \cdots \times r_d}$ and a set of factor matrices $\{\M{Q}_{i} \in \mathbb{R}^{n_i \times r_i}\}_{i=1}^d$, such that $\T{X} \approx \T{G} \times_1 \M{Q}_{1} \times_2 \M{Q}_{2} \cdots \times_d \mathbf{Q}_{d}$. If the chosen multilinear ranks $(r_1, \ldots, r_d)$ satisfy $r_i < \operatorname{rank}(\M{X}_{(i)})$ for some or all $i$, then the Tucker decomposition provides a low-rank approximation of $\T{X}$ in Tucker form. Higher order SVD (HOSVD) \cite{de2000multilinear} and sequentially truncated HOSVD (ST-HOSVD) \cite{vann2012new} are widely used direct algorithms for computing Tucker decompositions.  HOSVD performs singular value decomposition on each mode-$i$ unfolding $\M{X}_{(i)}$ of the tensor independently and selects the leading $r_i$ left singular vectors to form the orthonormal factor matrices $\M{Q}_i$. The core tensor is then computed by projecting the original tensor onto the subspaces defined by these matrices, i.e., $\T{G} = \T{X} \times_1 \M{Q}_1^T \times_2 \cdots \times_d \M{Q}_d^T$, see Algorithm~\ref{alg:hosvd}.

\begin{algorithm}[!ht]
    \caption{Higher Order SVD (HOSVD)~\cite{de2000multilinear} }
    \begin{algorithmic}[1]
    \REQUIRE Tensor $\T{X}$ and target rank $\V{r} = (r_1,r_2,\dots,r_d)$
    \FOR{$i=1,\dots,d$}
       \STATE Compute $\M{Q}_i$, $\M{Q}_i$ =  leading $r_i$ left singular vectors of $\M{X}_{(i)}$ 
    \ENDFOR
    \STATE Compute core $\T{G} = \T{X} \times_1 \M{Q}_1^T \times_2 \cdots \times_d \M{Q}_d^T$ 
    \RETURN Core tensor $\T{G}$, factor matrices $\{\M{Q}_i\}_{i=1}^d$    
    \end{algorithmic}
    \label{alg:hosvd}
\end{algorithm}

Unlike HOSVD, sequentially truncated HOSVD (ST-HOSVD) performs the truncation sequentially across modes by projecting the tensor onto lower-dimensional subspaces one mode at a time. This is done by computing a partially truncated core tensor via a tensor-times-matrix (TTM) product at each step, i.e., $\T{G} =\T{G} \times_i \M{Q}_i^T$, see Algorithm~\ref{alg:sthosvd}.
\begin{algorithm}[!ht]
    \caption{Sequentially Truncated HOSVD (ST-HOSVD)~\cite{vann2012new} }
    \begin{algorithmic}[1]
    \REQUIRE Tensor $\T{X}$ and target rank $\V{r} = (r_1,r_2,\dots,r_d)$
    \STATE Set $\T{G} = \T{X}$
    \FOR{$i=1,\dots,d$}
        \STATE Compute $\M{Q}_i$, $\M{Q}_i$ =  leading $r_i$ left singular vectors of $\M{X}_{(i)}$ 
        \STATE Update core $\T{G} \leftarrow \T{G} \times_i \M{Q}_i^\top$
    \ENDFOR
    \RETURN  Core tensor $\T{G}$, factor matrices $\{\M{Q}_i\}_{i=1}^d$    
    \end{algorithmic}
    \label{alg:sthosvd}
\end{algorithm}
 
\subsection{Probability background}
We review the background on probability and random matrices, that we will use frequently in our analysis, often without reference. 
\paragraph{Conditioning} Let $A,B$ be two events. The law of total probability says 
\[ \prob\{B\} = \prob\{B|A\}\prob\{A\} + \prob\{B|A^c\} \prob\{A^c \} \le \prob\{B|A\} + \prob\{A^c \}, \]
where $A^c$ denotes the complementary event. The lower bound $\prob\{B\} \ge \prob\{B|A\}\prob\{A\}$ will also be useful. 

\paragraph{$L^p$ norms} Given a random variable $X$, the  $L^p$ norms are defined as $\|X\|_{L^p} = (\expect |X|^p)^{1/p}$ for $p \geq 1$. We will need the centering identity $\|X - \expect X \|_{L^p} \le 2 \|X\|_{L^p}$ for any $p \geq 1$; see, for e.g.,~\changes{\cite[Lemma 49]{haselby2023fast}}. We also have $\|X^2\|_{L^p} = \|X\|_{L^{2p}}^2$.

We present a lemma that is essentially due to~\cite[Corollary 33]{haselby2023fast} but applies to non-symmetric random variables. We use the symmetrization inequalities to achieve this result. 
\begin{lemma}\label{lem:concineq}
 Let $Z_1,\dots, Z_\ell$ be a sequence of independent and identical mean zero random variables with $\|Z_i\|_{L^p} \le (Cp)^d$ for $1\le i \le \ell$, then for $p \ge 2$
 \[ \| \frac1\ell \sum_{i=1}^\ell Z_i \|_{L^p} \le (C')^d \max\left\{ \sqrt{\frac{p}{\ell}}, \frac{p^d}{\ell} \right\}, \]
 where $C'$ is a constant independent of $p,\ell$  but that depends on $C$.
\end{lemma}
\begin{proof}
Let $\varepsilon_1, \ldots, \varepsilon_{\ell}$ be independent Rademacher random variables that are also independent of $Z_1, \ldots, Z_{\ell}$. Then each $\varepsilon_i Z_i$ is symmetric and  and satisfies $\| \varepsilon_i Z_i\|_{L^p} = \|Z_i\|_{L^p} \leq (Cp)^d$. On applying symmetrization \cite[Lemma 6.3]{ledoux1991probability} together with  \cite[Corollary 33]{haselby2023fast}, we get
\[
\|\frac{1}{\ell}\sum_{i=1}^{\ell} Z_i\|_{L^p} \leq 2  \|\frac{1}{\ell}\sum_{i=1}^{\ell} \varepsilon Z_i\|_{L^p}  \leq  (C')^d \max\left\{ \sqrt{\frac{p}{\ell}}, \frac{p^d}{\ell} \right\}
\]
for $C' = 2eC$. 
\end{proof}

\paragraph{Sub-Gaussian random variables} We say that $X$ is an sub-Gaussian random variable~\cite[Definition 2.5.6]{vershynin2018high} if there exists a constant $K_1>0$ such that
\[ \prob\{ |X| \geq t\} \leq 2 \exp(-t^2/K_1^2),  \qquad \forall t > 0,\]
The sub-Gaussian norm of $X$, denoted $\|X \|_{\Psi_2}$, is defined as   
\[ \|X \|_{\Psi_2}  = \inf\left\{ t > 0 \, | \, \expect \exp(|X|^2/t^2) \le 2 \right\}.  \]

\subsection{\changes{Khatri-Rao product matrices}} As mentioned earlier, let $\kr$ denote the Khatri-Rao product (columnwise Kronecker product).
\begin{definition}[\krp]\label{def:subgauss}
\changes{We call $\M{\Omega}$ a KRP of order $d$ if } $\M{\Omega} = \M{\Omega}^{(1)} \kr \cdots \kr \M{\Omega}^{(d)}$, where $\M{\Omega}^{(j)} \in \R^{n_j\times \ell}$   for $1\leq j \leq d$ are independent random matrices whose entries are mean zero, unit variance,  subgaussian random variables, with subgaussian norm  bounded by $K \geq 1$.
\end{definition}
Throughout this paper, we will also use the notation 
\begin{equation}\label{eqn:N} N = \prod_{j=1}^d n_j.\end{equation}

This next result shows that, by construction, the columns of the random matrix are independent and isotropic. 
\begin{lemma}[Isotropic columns]
Let $\M{\Omega} = \V{\Omega}^{(1)} \kr \cdots \kr \V{\Omega}^{(d)}\in \R^{N \times \ell} $ be a \krp~as in Definition~\ref{def:subgauss}. The random matrix has independent and isotropic columns, i.e., $\expect[\V{\Omega}(:,k)] = \V{0}$ and $\expect[\V{\Omega}(:,k)\V{\Omega}(:,k)\t]  = \M{I}_{N}$ for $1 \leq k \leq \ell$.
\end{lemma}
\begin{proof}
    \changes{See~\cite[Lemma 4.1]{CW19}. }
\end{proof}

\changes{
Our analysis relies on the following lemma, taken from \cite{ahle2020oblivious}, which bounds the $L^p$ norms of quantities involving columns of KRP matrices.
\begin{lemma}{\cite[Lemma 4.9] {ahle2020oblivious}} \label{lem:ahle_etal}
    For $1 \le i\le d$, let $\M{\omega}_{i} \in \R^{n_i}$ be independent random vectors satisfying the Khintchine inequality $\|\langle \M{\omega}_i, \M{x}\rangle\|_{L^{p}} \leq C_p \|\M{x}\|_2$ for every $p \ge 1$ and any vector $\M{x} \in \R^{n_i}$. Then for any vector $\M{a} \in \R^{N}$ and $p\ge1$, $\|\langle \M{\omega}_1 \kron \cdots \kron \M{\omega}_d, \M{a}\rangle\|_{L^{p}} \leq C_{p}^d \|\V{a}\|_2$. 
\end{lemma}
}

\section{Randomized range finding}\label{sec:rrf}
We state the main result of the analysis of the randomized range finder in \Cref{ssec:rrfmain} and give the proof in \Cref{ssec:rrfproof}.
\subsection{Main result}\label{ssec:rrfmain}Let $\M{M}\in \R^{M \times N}$ be a matrix, for which we want to compute a low-rank approximation with target rank $r\le \min\{M,N\}$. 
We consider the partitioned SVD of $\M{M}$ as 
\[ \M{M} = \bmat{\M{U}_r & \M{U}_\perp} \bmat{ \M\Sigma_r & \\ & \M\Sigma_\perp} \bmat{\M{V}_r\t \\ \M{V}_\perp\t}.   \]
where $\M{U}_r\in \R^{M\times r}$, $\M\Sigma_r \in \R^{r\times r}$, and $\M{V}_r\in\R^{N \times r}$.

 We derive a result on the randomized range finder (\cref{alg:rrf}) \changes{with KRP}. 
 
\begin{theorem}\label{thm:rrf}
Let $\M{Q}$ be the result of the randomized range finder of $\M{M} \in \R^{ M\times N}$ with input \krp{} ~$\M{\Omega} = \M{\Omega}^{(1)} \kr \dots \kr \M{\Omega}^{(d)} \in \R^{N \times \ell}$. Let parameters $0 < \delta \le 2e^{-2}$, and let $\ell$  satisfy $\ell \ge C^d \ln^d(2\cdot 9^r/\delta)$. With probability at least $1-\delta$, 
\[ \begin{aligned}\| \M{M} - \M{QQ}\t\M{M} \|_F^2 \leq & \>   (1 + 2 (1+\Gamma)) \|\M\Sigma_\perp\|_F^2,
\end{aligned}\]
where \[\Gamma \equiv   (C')^d\max\left\{\frac{\ln^{1/2}(2/\delta)}{\ell^{1/2}},  \frac{\ln^d(2/\delta)}{\ell} \right\}.\]
Here $C,C'$ are constants that only depend on $K$. 
\end{theorem}

 \changes{A recent paper~\cite{camano2025faster} also derives results for the randomized range finder using \krp{} matrices. Their proof relies on showing that \krp{} matrices satisfy a weaker notion of subspace embedding called subspace injections, which is sufficient for range finder and other low-rank approximation results. The number of samples $\ell$ required by their analysis~\cite[Theorem 5.1 and Theorem 1.3]{camano2025faster} is $\mc{O}(C^d r)$, where $C$ is an absolute constant. A possible benefit of our approach is that their result assumes a fixed failure probability $\delta$, whereas our result tracks the failure probability in the number of samples and the error estimate for the range finder. In fact for $d=1$, it gives the same dependence on failure probability as a matrix with independent subgaussian entries~\cite[Theorem 3.3]{saibaba2025randomized}. Our analysis does not assume $n_1 = \dots = n_d$ as~\cite{camano2025faster} apppears to do, so is slightly more general, though this doesn't seem to be a major limitation in their analysis. On the other hand, their analysis applies to far more distributions than is covered by our analysis.  } The analysis of the range finder using \krp{}s has also been done in the context of Tucker low-rank approximations~\cite{CW19,che2025efficient}. Therefore, we include a comparison with that existing work in \Cref{sec:tucker}.

Finally, it is worth remembering that the constants show an exponential dependence on $d$, suggesting that the analysis may only be meaningful for small or moderate dimensions (i.e., $d = \changes{2}$ to $5$). Extensive numerical evidence for $d=2$ in~\cite{bujanovic2024subspace} suggest that \krp{}~matrices perform only slightly worse than Gaussian random matrices. Numerical experiments in this paper also suggest that the theoretical upper bounds are pessimistic in practice. An  exponential dependence on $d$ appears to be necessary based on fundamental results by~\cite{meyer2026hutchinson,meyer2025understanding} on Kronecker matrix-vector product complexity.

\subsection{Proof of \texorpdfstring{\cref{thm:rrf}}{Theorem 3.1}}\label{ssec:rrfproof}
This subsection is dedicated to the proof of \cref{thm:rrf}. Before we launch into the proof, we need several intermediate results. It should be noted that the constants in the analysis may change from step to step.

\begin{lemma}\label{lem:kronrank1new} Let $\V{\omega} \in \R^{N\times 1}$ be a \krp{} (see Definition~\ref{def:subgauss}) and let $\M{V}\in \R^{N \times k}$ be an arbitrary {non-zero} matrix with columns $\V{v}_j$ for $1\le j \le k$. Then, 
\[ \|\|\M{V}\t \V{\omega}\|_2^2 - \|\M{V}\|_F^2\|_{L^p} \le \|\M{V}\|_F^2 (Cp)^d,  \]
where $C$ is a constant that depends only on $K$. 
\end{lemma}

\begin{proof}
The proof has two steps.
\paragraph{1. Defining the random variable $Z$} Without any loss of generality, assume that  $\|\M{V}\|_F = 1$. Now, consider the random variable $$Z =  \|\M{V}\t\V{\omega}\|_2^2 - \|\M{V}\|_F^2 = \sum_{j=1}^k ((\V{v}_j\t\V{\omega})^2-\|\V{v}_j\|_2^2) ,$$ which has zero mean. In the next step, we will analyze the properties of $Z$. 

\paragraph{2. Bounding $\|Z\|_{L^p}$} For every $1 \le j \le d$, by the assumptions on $\V{\omega}_j$, by~\cite[Lemma 3.4.2]{vershynin2018high}, \changes{for any  vector $\V{y} \in \R^{n_j}$}, the random variable $\V{\omega}_j\t\V{y}$ is subgaussian with norm  $\|\V{\omega}_j\t\V{y}\|_{\Psi_2}\le \changes{\bar{C}_S} K \|\V{y}\|_2$, where $\changes{\bar{C}_S}$ an absolute constant. Then, by~\changes{\cite[Eq. (2.15) in Section 2.5.2]{vershynin2018high}}, for $p\geq 1$, we have
\[ \| \V{\omega}_j\t\V{y}\|_{L^p} \le c \,  \sqrt{p}\,  \| \V{\omega}_j\t\V{y}\|_{\Psi_2} = C_SK \sqrt{p} \|\V{y}\|_2 
\] 
where $C_S =c \cdot \bar{C}_S > 0$ is some absolute constant.
 By \changes{Lemma \ref{lem:ahle_etal}},  for any  vector $\V{z} \in \R^N$, we have 
\begin{equation} \label{eqn:krd}\| \V{\omega}\t\V{z}\|_{L^p} \le (C_S K \sqrt{p})^d \|\V{z}\|_2, \qquad \forall p \ge 1.   \end{equation}
So, by the triangle inequality and the centering property of $L^p$ norms, 
\[ \begin{aligned}\|Z\|_{L^p}  \le  & \> 2\sum_{j=1}^k \|(\V{v}_j\t\V{\omega})^2\|_{L^p} \\
\le &  2\sum_{j=1}^k \|\V{v}_j\t\V{\omega}\|_{L^{2p}}^2 \le 2(\changes{C_S} K\sqrt{2p})^{2d} \left(\sum_{j=1}^k \|\V{v}_j\|_2^2 \right).   \end{aligned}\]
The last inequality follows from~\eqref{eqn:krd} and note that $\|\M{V}\|_F^2 = \sum_{j=1}^k \|\V{v}_j\|_2^2 =1$. Therefore, $\|Z\|_{L^p} \le 2(C_SK\sqrt{2})^{2d}  p^d \leq C^d p^d$ for $C = 4 \, C_S^2 \,K^2$.  This completes the proof.
\end{proof}

A matrix $\M\Omega$ is said to be an oblivious subspace embedding for the subspace $\mathcal{R}(\M{W})$ with distortion parameter $\epsilon \in (0,1)$ if 
\[ (1-\epsilon)\|\V{x}\|_2  \le \|\M\Omega \V{x}\|_2 \le (1+\epsilon)\|\V{x}\|_2 \qquad \forall \V{x} \in \mathcal{R}(\M{W}). \]
The next result establishes an oblivious subspace embedding for \krp{} matrices. The authors in~\cite{ahle2020oblivious} and~\cite{haselby2023fast} derived bounds for Johnson-Lindenstrauss type bounds involving \krp{} matrices. From here, one can convert it into a subspace embedding using a standard epsilon-net argument, which we present here. We can also infer this result from combining~\cite[Theorem 28 and Lemma 10]{haselby2023fast}, but we give a bit more detail in our exposition. We denote the range of a matrix $\M{A}$ by $\mathcal{R}(\M{A})$.

\begin{theorem}\label{thm:subspace2}
    Let  $\M{W} \in \R^{N \times r}$ have orthonormal columns and $\M{X} \in \R^{N \times \ell}$ be \krp{} (Definition~\ref{def:subgauss})  with $$\ell \ge C^d \ln^d(9^r/\delta) , $$
    where $C$ is a constant that only depends on the parameter $K$. Then, for any parameter $0 < \delta < 1$, with probability at least $1-\delta$, the following bounds hold independently 
    \begin{align} \label{eqn:pseudo} \frac{2}{3\ell} \le & \>  \| (\M{W}\t\M{X})^\dagger\|_2^2 \le \frac{2}{\ell} \\ \label{eqn:krpose} 
   \frac12\|\V{x}\|_2^2 \le & \> \| \frac1\ell \M{X}\t \V{x}\|_2^2 \le \frac32 \|\V{x}\|_2^2 \qquad \forall \V{x} \in \mathcal{R}(\M{W}) .    \end{align}

\end{theorem}
\begin{proof}
This result is based on a standard $\epsilon$-net argument similar to~\cite[Theorem 4.6.1]{vershynin2018high}.  Let $\mc{N}$ be a $1/4$-net on the unit sphere $S^{r-1}$ with cardinality bounded by $9^r$. We want to analyze the singular values of $\M{Z} = \M{X}\t\M{W} \in \R^{\ell \times r}$ which has independent rows. By~\cite[Exercise 4.4.3]{vershynin2018high}
    \begin{equation}\label{eqn:epsilonnet2}\| \frac{1}{\ell}\M{Z}\t\M{Z}-\M{I}\|_2 \le 2 \max_{\V{c}\in \mc{N}} | \frac{1}{\ell} \| \M{Z}\V{c}\|_2^2 - 1|. \end{equation}
    Define the random variables $X_i = 
    \left((\M{X}(:,i)\t\M{Wc})^2-1\right)$ for $1\le i \le \ell$ so that {$\frac1\ell\|\M{Z}\V{c}\|_2^2 -1 = \frac{1}{\ell}\sum_{i=1}^\ell X_i$}. Each random variable $X_i$ is independent, has zero mean and satisfies $\|X_i\|_{L^p} \le (Cp)^d$ for some $C$ that only depends on $K$. We have used Lemma~\ref{lem:kronrank1new} with $\M{V} = \M{W}\V{c}$ and $\|\M{W}\V{c}\|_2 = 1$.
 
Thus, by Lemma~\ref{lem:concineq}, 
    \[ \| \frac1\ell \sum_{i=1}^\ell X_i \|_{L^p} \le C^d \max\left\{ \sqrt{\frac{p}{\ell}}, \frac{p^d}{\ell} \right\}.\]
    We can use Markov inequality to get the tail bound 
    \[ \prob\{ |\frac1\ell \sum_{i=1}^\ell X_i| \ge t \} = \prob\{ |\frac1\ell \sum_{i=1}^\ell X_i|^p \ge t^p \} \le  \frac{C^{dp}}{t^p} \max\left\{ \left(\frac{p}{\ell}\right)^{p/2}, \frac{p^{pd}}{\ell^p} \right\} .\] 
    Taking a union bound over the net $\mc{N}$ and using~\eqref{eqn:epsilonnet2}
    \[ \prob\{  \| \frac{1}{\ell}\M{Z}\t\M{Z}-\M{I}\|_2 \ge 2t\} \le 9^r\frac{C^{dp}}{t^p} \max\left\{ \left(\frac{p}{\ell}\right)^{p/2}, \frac{p^{pd}}{\ell^p} \right\}. \] 
    We take $t = 1/4$, $p = \ln(9^r/\delta)$ and note that $p \ge 2$ for any $r \ge 1$ and $\delta < 1$. Set the failure probability to $\delta$ and solve for $\ell$ to get 
   \[ \ell \ge \max\left\{ {16 e^2 \cdot C^{2d}}\ln(9^r/\delta), {4e \cdot C^d}\ln^d(9^r/\delta) \right\}. \]
    Adjusting the constants, we get the bound for $\ell$ in the statement of the theorem and 
    $\prob\{ \| \frac{1}{\ell}\M{Z}\t\M{Z}-\M{I}\|_2 \ge \frac12\} \le \delta$. Denote the smallest singular value of $\M{Z}$ as $\sigma_r$, so with probability at least $1 -\delta$,  $|\sigma_r^2/\ell - 1| \le 1/2$ and $ 2/(3\ell) \le \sigma_r^{-2} \le 2/\ell$. Finally, $\|\M{Z}^\dagger\|_2^2 = 1/\sigma_r^2$ which completes the proof of the first set of inequalities~\eqref{eqn:pseudo}. 

    For the second set of inequalities~\eqref{eqn:krpose}, let $\sigma_1$ be the largest singular value of $\M{Z}$. So, $\| \frac{1}{\ell}\M{Z}\t\M{Z}-\M{I}\|_2 $ implies $ \frac12 \le \sigma_r^2/\ell \le \sigma_1^2/\ell \le \frac32$ holds with probability at least $1-\delta$. The second set of inequalities~\eqref{eqn:krpose} follow readily from the variational characterization of extremal singular values.  
    \end{proof}

{
\paragraph{Discussion} 
Theorem~\ref{thm:subspace2} shows that the matrix $\M\Omega = \M{X}\t/\sqrt{\ell}$ is an oblivious subspace embedding for the subspace $\mathcal{R}(\M{W})$ with distortion parameter $\epsilon= 1 - \sqrt{1/2}$, with probability at least $1-\delta$ if $\ell$ is sufficiently large. We compare our results for subspace embedding with other known results. For the comparisons with other subspace embedding results, we consider a fixed $\epsilon$. Our analysis predicts the number of samples to be $\mc{O}(C^d(r^d + \ln^d(1/\delta)))$.

The analysis for Gaussian random matrices predicts the number of samples to be $\mc{O}(r + \ln(1/\delta) )$, which is substantially smaller than for \krp{}~matrices. However, it is worth recalling that Gaussian random matrices require much more effort to generate. More generally,~\cite{saibaba2021efficient} showed that  for the class of random matrices $\M{\Omega}$ with subgaussian rows (which includes the independent entries as a special case), the number of samples for a subspace embedding is $\mc{O}(r + \ln(1/\delta) )$.  For subsampled randomized Hadamard transform, as shown in~\cite{tropp2011improved,martinsson2020randomized}, the number of samples is $\mc{O}(r \ln r)$. Corresponding results for other distributions are given in~\cite{saibaba2025randomized,martinsson2020randomized}.

We now turn our attention to comparison with other \krp{}~results. The analysis for \krp{}~matrices was done for the case $d=2$ in~\cite{bujanovic2024subspace}. The number of samples for a subspace embedding is
\[ \ell \sim r^{3/2}  + r\ln(1/\delta) + r^{1/2}\ln^2(1/\delta). \] 
For $d=2$ and some range of $r$ this result is more favorable than our result,  which predicts $\ell = \mc{O}(r^2 + \ln^2(1/\delta))$.  It should be noted, however, that our results apply to \krp{}s with any subgaussian entries, not just the Gaussian distribution considered in~\cite{bujanovic2024subspace}. }

We now continue with our proof of Theorem~\ref{thm:rrf}.
\begin{theorem}\label{thm:zbound}
    Let  $\M{Z} \in \R^{M \times N}$ be nonzero and $\M{X} \in \R^{N \times \ell}$ be \krp{} (Definition~\ref{def:subgauss}). Then, for {$0 < \delta \le e^{-2}$}, with probability at least $1-\delta$,
    \[  \| \M{ZX}\|_F^2  \le  \ell \|\M{Z}\|_F^2(1 + \Gamma) \quad \text{ where } \quad \Gamma \equiv  C^d\max\left\{\frac{\ln^{1/2}(1/\delta)}{\ell^{1/2}},  \frac{\ln^d(1/\delta)}{\ell} \right\} ,  \]
    and $C$ is a constant that only depends on the sub-Gaussian norm $K$.  
\end{theorem}
\begin{proof}
    Without loss of generality, let $\|\M{Z}\|_F= 1$. Define the independent, mean zero, random variables 
    \[ Z_i = \| \M{ZX}(:,i)\|_2^2 - \|\M{Z}\|_F^2 \qquad \text{ for} \quad 1 \le i \le \ell,   \]
    so that  $\| \M{ZX}\|_F^2 - \ell \|\M{Z}\|_F^2 = \sum_{i=1}^\ell Z_i$. Applying Lemma~\ref{lem:kronrank1new}, we have $\|Z_i\|_{L^p} \le (Cp)^d$ for $1 \le i \le \ell$. By Lemma~\ref{lem:concineq}, 
    \[ \| \frac1\ell \sum_{i=1}^\ell Z_i\|_{L^p} \le C^d \max\left\{ \sqrt{\frac{p}{\ell}}, \frac{p^d}{\ell} \right\}.   \]
     We can use Markov inequality to get the tail bound 
    \[ \prob\{ |\frac1\ell \sum_{i=1}^\ell Z_i| \ge t \} = \prob\{ |\frac1\ell \sum_{i=1}^\ell Z_i|^p \ge t^p \} \le  \frac{C^{dp}}{t^p} \max\left\{ \left(\frac{p}{\ell}\right)^{p/2}, \frac{p^{pd}}{\ell^p} \right\} .\] 
    Set the tail bound to $\delta$ and $p = \ln(1/\delta)$. By the assumption on $\delta$, $p \ge 2$. Solve for $t$ to get
    {\[ t = e \, C^d\max\left\{\left(\frac{\ln(1/\delta)}{\ell}\right)^{1/2},  \frac{\ln^d(1/\delta)}{\ell} \right\}. \]}
    Adjusting the constants, we get 
    $\prob\{ | \frac1\ell \|\M{ZX}\|_F^2 - 1| \ge \Gamma \} \le \delta$. 
    This completes the proof.  
\end{proof}

At last, we are ready to give a proof of Theorem~\ref{thm:rrf}. The proof of this theorem follows quickly from all the intermediate results.
\begin{proof}[Proof of Theorem~\ref{thm:rrf}]
    
    From Theorem~\ref{thm:subspace2} with $\M{W} = \M{V}_r$, we have $\|(\M{V}_r\t\M{\Omega})^\dagger\|_2^2\le 2/\ell$ with probability at least $1-\delta/2$. Thus, for this event $\text{rank}(\M{V}_r\t\M\Omega) = r$. Applying the deterministic bound~\cite[Theorem 9.1]{halko2011finding} on this event 
\begin{equation}\label{eqn:hmt2} \| \M{M}-\M{QQ}\t\M{M}\|_F^2 \leq \|\M\Sigma_\perp\|_F^2 +  \|\M\Sigma_\perp (\M{V}_\perp\t \M{\Omega})(\M{V}_r\t\M{\Omega})^\dagger\|_F^2. \end{equation}
Applying submultiplicativity 
\begin{equation}\label{eqn:rrfinter} \|\M\Sigma_\perp (\M{V}_\perp\t \M{\Omega})(\M{V}_r\t\M{\Omega})^\dagger\|_F^2 \le \|\M\Sigma_\perp (\M{V}_\perp\t \M{\Omega})\|_F^2 \|(\M{V}_{\changes{r}}\t\M{\Omega})^\dagger\|_2^2 . \end{equation}
By Theorem~\ref{thm:zbound} with 
$\M{Z} = \M\Sigma_\perp\M{V}_\perp\t$, we have $\|\M{ZX}\|_F^2 \le \ell \|\M{\Sigma}_\perp\|_F^2 (1+\Gamma)$, where $\Gamma$ is given in the statement of the theorem, with probability at least $1-\delta/2$. By a union bound, with probability at least $1-\delta$,  $\|\M\Sigma_\perp (\M{V}_\perp\t \M{\Omega})(\M{V}_r\t\M{\Omega})^\dagger\|_F^2 \le 2\|\M{\Sigma}_\perp\|_F^2 (1+\Gamma)$. Plug this bound into~\eqref{eqn:hmt2} to complete the proof.
\end{proof}

\section{Application: Compressing Block-structured matrices}\label{sec:blockstruct}
In this section, we consider a block matrix $\M{M} \in \R^{(mp)\times (nq)}$ containing $p \times q$ blocks of size $m\times n$ each. In many applications, the matrices are structured at the block level (e.g., block-Toeplitz, block-Hankel matrices) and this structure can be exploited to develop efficient algorithms. Following the approach in~\cite{kilmer2021structured}, we can express the block-structured matrix in the form 
\[ \M{M} = \sum_{j=1}^t \M{E}_j \otimes \M{M}_j, \]
where $t$ is the number of distinct block matrices $\{\M{M}_j\}_{j=1}^t$ and $\M{E}_j \in \R^{p\times q}$ for $1\le j \le t$ are  matrices with ones and zeros representing the locations of the nonzero blocks.\footnote{This is slightly different from~\cite{kilmer2021structured} in which the blocks $\M{E}_j$ are scaled by the number of times that the subblocks appear.} 
In applications such as the one described in \Cref{sec:SI}, we seek to compute a low-rank approximation of $\M{M}$. 
Since the matrices can be quite large, the computation of the SVD is infeasible, so randomized algorithms are cheaper and more feasible. 
The standard randomized SVD is computationally advantageous over the SVD; however, it has two drawbacks (1) it does not naturally exploit the structure of the matrix, and (2) generating and storing the random matrices can be challenging when the matrix size is large. 

\subsection{Efficient compression algorithm} 
The first approach we could consider is applying randomized SVD (\cref{alg:rsvd}) with $\M{\Omega}$ as a \krp{}. 
The sketch $\M{M\Omega}$ can be computed efficiently, but the step in \cref{line:secondpass} requires $\M{Q}\t\M{M}$. 
This is not as efficient, as $\M{Q}$ no longer has a Kronecker product structure.  
Therefore, to address this, we use a single-\changes{pass} approach for computing a low-rank approximation, as in~\cite{tropp2017practical}.

Suppose the target rank is $r$. We draw a random matrix $\M{\Omega} = \M{\Omega}^{(1)} \kr \M{\Omega}^{(2)} \in \R^{(nq )\times \ell_r}$, where the factors have independent Gaussian entries with zero mean and variance $1$. 
The parameter $\ell_r$ represents the sketch size, which is larger than the target rank. Next, we compute the sketch $\M{Y} = \M{M}\M{\Omega}$; this can be efficiently done using the properties of Kronecker products, as
\begin{equation}\label{eqn:blockstruct}
    \M{Y} = \M{M}\M{\Omega} =  \left(\sum_{j=1}^t \M{E}_j \otimes \M{M}_j \right)(\M{\Omega}^{(1)} \kr \M{\Omega}^{(2)}) = \sum_{j=1}^t (\M{E}_j\M{\Omega}^{(1)})\kr (\M{M}_j\M{\Omega}^{(2)}).
\end{equation}  
We then compute a thin-QR factorization $\M{Y} = \M{QR}$. 
The intuition is that $\M{M} \approx \M{QQ}\t\M{M}$ as captured in the statement of \cref{thm:rrf}. Another sketch $\M{Z} = \M{M}\t\M{\Psi}$ is computed by a similar formula as in~\eqref{eqn:blockstruct}, where $\M{\Psi} = \M{\Psi}^{(1)}\kr \M{\Psi}^{(2)} \in \R^{(mp)\times \ell_{l}}$. We then obtain the low-rank approximation
\begin{equation}
    \label{eqn:singleview}
    \M{M} \approx \M{Q} (\M{\Psi}\t\M{Q})^\dagger \M{Z}\t.
\end{equation}
The low-rank approximation~\eqref{eqn:singleview} can be further post-processed by converting to SVD format and truncating to retain the $r$ largest singular values. The details are given in Algorithm~\ref{alg:blockstructured}. Guidance for $\ell_r$ and $\ell_l$ can be found in Section~\ref{ssec:blockstructanalysis}. 

The computational cost of computing the sketches is $$\mc{O}\left( rt(mp + nq) + \sum_{j=1}^t r(\text{nnz}(\M{E}_j) + \text{nnz}(\M{M}_j) ) \right) \> \text{flops},$$
assuming $\ell_1, \ell_2 = \mc{O}(r)$. 
There is an additional cost of $\mc{O}( (mp + nq) r^2 )$ flops for orthogonalizing and post-processing the low-rank approximation. 
This approach is advantageous over the same algorithm but with Gaussian random matrices because the random matrix takes advantage of the structure of $\M{M}$ while computing matrix-vector products. 
More specifically, if $\M{\Omega}$ and $\M{\Psi}$ are standard Gaussian random matrices, the cost of forming the sketches is $\mc{O}( \sum_{j=1}^t r\cdot( (m+n)\text{nnz}(\M{E}_j) + (q+p)\text{nnz}(\M{M}_j) )$ flops. The cost of orthogonalizing and post-processing is the same as before. Furthermore, the cost of generating and storing the random matrix is reduced considerably. We only need to generate $(n+q)\ell_r + (m+p)\ell_l$ random numbers, rather than $nq\ell_r + mp\ell_l$ random numbers for standard Gaussian random matrices.

\begin{algorithm}[!ht]
\begin{algorithmic}[1]
\REQUIRE Matrix $\M{M} \in \R^{(mp)\times (nq)}$, target rank $r$, parameters $\ell_l,\ell_r$.
\STATE Draw matrix  $\M{\Omega} = \M{\Omega}^{(1)} \kr \M{\Omega}^{(2)} \in \R^{(nq )\times \ell_r}$, $\M{\Psi} = \M{\Psi}^{(1)}\kr \M{\Psi}^{(2)} \in \R^{(mp) \times \ell_l}$
\STATE Compute sketches $\M{Y} = \sum_{j=1}^t (\M{E}_j\M{\Omega}^{(1)})\kr (\M{M}_j\M{\Omega}^{(2)})$ and $\M{Z} = \sum_{j=1}^t (\M{E}_j\t\M{\Psi}^{(1)})\kr (\M{M}_j\t\M{\Psi}^{(2)})$
\STATE Compute thin QR factorization $\M{Y} = \M{QR}$
\STATE Compute $\M{W} = (\M{\Psi}\t \M{Q})^\dagger \M{Z}\t$
\STATE Compute the thin SVD $\M{U}_W\M\Sigma\M{V}\t =\M{W}$
\STATE Compute $\M{U} = \M{QU}_W$
\RETURN Matrices $\M{U}(:,1:r), \M\Sigma(1:r,1:r), \M{V}(:,1:r)$ that can be used to form a low-rank approximation to $\M{M}$
\end{algorithmic}
\caption{Single \changes{pass} Randomized SVD for block-structured matrices}
\label{alg:blockstructured}
\end{algorithm}

If $\M{M}$ is symmetric positive semidefinite, we can instead use the Nystr\"om approximation
\[ \M{M} \approx \M{M}_{\rm nys} :=  \M{M\Omega} (\M{\Omega}\t\M{M\Omega})^\dagger (\M{M\Omega})\t,\]
where $\M{\Omega} = \M{\Omega}^{(1)}\kr \M{\Omega}^{(2)}$ is the Khatri-Rao product of the sketches.

\subsection{Numerical illustration: System identification} 
\label{sec:SI}
In this application, we consider the problem of system identification using the  eigensystem realization algorithm (ERA). The goal is to obtain the matrices $(\M{A},\M{B},\M{C},\M{D})$ that govern the system dynamics from the data, which involves the so-called Markov parameters $\M{H}_k = \M{CA}^{k-1}\M{B} \in \R^{m\times n}$ for $1 \leq k \leq 2s-1$ and $\M{H}_0 = \M{D}$. Here, $s$ determines the number of time measurements. Since the first Markov parameter $\M{H}_0 = \M{D}$, estimating $\M{D}$ is trivial. To recover the other three matrices, the steps involved are to form a large block-Hankel matrix 
\[ \M{\mathcal{H}} = \bmat{\M{H}_1 & \M{H}_2  &  \dots & \M{H}_{s-1}\\ \M{H}_2 & \M{H}_3 & \dots & \M{H}_{s} \\ \dots & \dots & \dots & \dots \\\M{H}_{s-1} & \dots& \dots  & \M{H}_{2s-1}  } \in \R^{(ms)\times (ns)}. \]
{The expensive step is to then compute a rank-$r$ approximation to the matrix. The resulting singular triplets are then used to solve a least-squares problem to identify the system matrices. The details are given in~\cite{minster2021efficient}.} In this paper, we focus on accelerating the compression step. Previous work~\cite{minster2021efficient,kilmer2021structured} has developed randomized algorithms for computing the dominant singular triplets, which exploited the block-Hankel structure in other ways. 

In contrast,  we follow a different approach. We first note that the block-structured matrix has the decomposition
\[ \M{\mathcal{H}} = \sum_{k=1}^{2s-1} \M{E}_k\kron \M{H}_k,  \]
where $\M{E}_k\in \R^{s\times s}$ is a Hankel matrix with $1$'s corresponding to the entries satisfying $i+j-1 = k$, and $0$'s everywhere else. We exploit this decomposition along with the \krp{} structure. 

We consider the same test problem as in~\cite[Section 5.2]{minster2021efficient} with $m=155$, $n=50$, and $s=200$; in this instance, the problem size is $31,000\times 10,000$. 
We apply single-\changes{pass} randomized SVD for block-structured matrices with dense Gaussian random matrices and KRP structured random matrices (\cref{alg:blockstructured}) to compute the rank-$r$ SVD with $r=155$, $\ell_r = r+\rho$, and $\ell_l = \lceil 1.5\ell_r\rceil$ and then apply the remaining steps of the ERA algorithm. We refer to the single-\changes{pass} randomized SVD method using dense Gaussian matrices as R-Gauss and the one using KRP-structured matrices as R-KRP. For comparison, we use the Randomized ERA algorithm from~\cite[Section 3.1]{minster2021efficient}, which uses a Gaussian random matrix to compute the Randomized SVD (which we denote RandERA). Note that this algorithm also exploits the block-Hankel structure. In our experiments, we use the oversampling parameter $\rho=20$. To evaluate these methods, we performed run-time and accuracy experiments for different rank values $r$. We assess accuracy in terms of the Hausdorff distance~\cite[Equation (5.1)]{minster2021efficient} between the eigenvalue set of the ground-truth matrix $\M{A}$ and the eigenvalue set of its reconstruction via R-KRP, R-Gauss, and RandERA. A smaller Hausdorff distance indicates a more accurate recovery.

In the left panel of \cref{fig:krpera}, we summarize the average running time results over 10 runs, and in the right panel, the average Hausdorff distance results. From the running time plots, we observe that the running time of R-KRP is the lowest among all methods across all rank values $r$. The accuracy results in the right panel show that the average Hausdorff distance error of R-KRP and R-Gauss are similar, and are either similar or better than that of RandERA. 

A key advantage of R-KRP is that it reduces both computational complexity and the overhead of random number generation. In these experiments, R-KRP requires only approximately $7.8\%$ of the random numbers used by RandERA and about $1.3\%$ of those used by R-Gauss, while maintaining comparable or superior performance compared to RandERA. We note around 4.5$\times$ speedup over R-Gauss and a 3$\times$ speedup over RandERA.

Finally, we note that RandERA is specific to block-Hankel systems whereas R-KRP is applicable to block-structured matrices, more broadly. 

\begin{figure}[!ht]
    \centering    \includegraphics[width=0.92\textwidth]{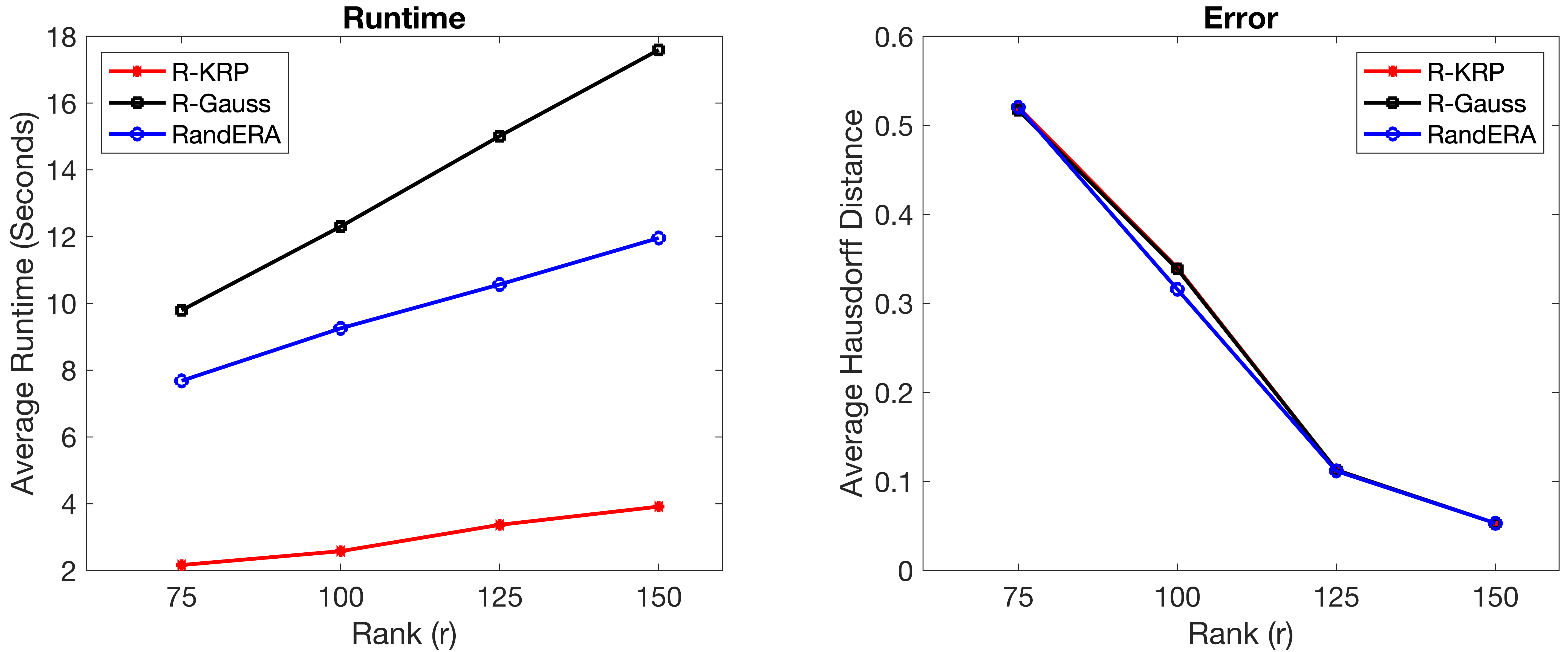}
    \caption{Left: Average running time over 10 runs for various methods used to recover the system matrices. Right: Average recovery error, measured by the Hausdorff distance between the eigenvalue set of the ground-truth matrix $\M{A}$ and the eigenvalue set of its reconstruction using R-KRP, R-Dense, and RandERA methods.}
    \label{fig:krpera}
\end{figure}

\subsection{Extensions to higher-order structure} This approach can be trivially extended to compute low-rank approximations of matrices that can be expressed as sums of Kronecker products. That is, assuming, $\M{M} = \sum_{j=1}^t\M{B}_j \otimes \M{C}_j$. In many applications, the matrix may have multilevel structure (e.g., block Toeplitz with Toeplitz blocks, or triply block Toeplitz). More precisely, let $\M{A} \in \R^{M\times N}$ with $L$ levels of block structure, where $M = m \left(\prod_{j=1}^L\ell_j\right)$ and $N = n \left(\prod_{k=1}^L q_j\right)$. Following~\cite[Section 7]{kilmer2021structured}, this can be expressed as 
\[ \M{M} = \sum_{i_1=1}^{p_1} \cdots \sum_{i_L=1}^{p_L} \M{E}_{i_1}^{(1)} \otimes \cdots \otimes
\M{E}_{i_L}^{(L)} \otimes \M{M}^{(i_1,\dots,i_L)} \]
where $\M{M}^{(i_1,\dots,i_L)}$ are the $m \times n$ non-redundant blocks at level $L$. The matrices $\M{E}_{k}^{(j)} \in \R^{\ell_j\times q_j}$, with entries in $\{0,1\}$, represent the different mapping matrices at level $j$, where $1 \leq j \leq L$. At each level $j$, there are $p_j$ such matrices where $1 \leq p_j \leq \ell_jq_j$. It is relatively straightforward to extend the approach in Algorithm~\ref{alg:blockstructured} to the multilevel structured case. The details are omitted. We provide an analysis for single-\changes{pass} approximation in the general higher order block structured case. 

\subsection{Analysis}\label{ssec:blockstructanalysis}
In this section, we analyze the single-\changes{pass} algorithm for a matrix $\M{M} \in \R^{m^d \times n^d}$ with target rank $r$. 
We draw two independent matrices $\M{\Omega} \in \R^{n^d \times \ell_r}$ and $\M{\Psi} \in \R^{m^d \times \ell_l}$. 
Next, we compute the sketches $\M{Y} = \M{M \Omega} $ and compute its thin QR factorization $\M{Y} = \M{QR}$. The resulting low-rank approximation to $\M{M}$ is then $\M{Q} (\M{\Psi}\t\M{Q})^\dagger \M{\Psi}\t\M{M}$. The following result provides analysis for the low-rank approximation. 

\begin{theorem}
    Let $\M{M} \in \R^{m^d \times n^d}$ and let $r$ be the target rank.  Let $\M{\Omega} \in \R^{n^d \times \ell_r} $ and $\M{\Psi} \in \R^{m^d \times \ell_l}$ be two independent \krp{} matrices (see Definition~\ref{def:subgauss}). Fix $0 < \delta < 4e^{-2}$, take the sketch sizes $\ell_r \ge  \> C_r^d \ln(4\cdot 9^r/\delta)$ and $\ell_l \ge \> C_l^d \ln(4 \cdot 9^{\ell_r}/\delta)$. 
    Then, with probability at least $1-\delta$
    \[ \|\M{M} - \M{Q}(\M{\Psi}\t\M{Q})^\dagger\M{\Psi}\t\M{M}\|_F^2 \le (1+2(1+2\Gamma_r) ( 1 + 2(1+\Gamma_l) ) \|\M\Sigma_\perp\|_2^2  ,  \]
    where 
    \[ \begin{aligned} \Gamma_r  \equiv &   (C_r')^d\max\left\{\frac{\ln^{1/2}(4/\delta)}{\ell_r^{1/2}},  \frac{\ln^d(4/\delta)}{\ell_r} \right\}  \\ \Gamma_l \equiv    &   (C_l')^d\max\left\{\frac{\ln^{1/2}(4/\delta)}{\ell_l^{1/2}},  \frac{\ln^d(4/\delta)}{\ell_l} \right\}  .\end{aligned}\]
\end{theorem}
Here $C_r, C_l, C_r'$ and $C_l'$ are constants that only depend on $K$. 
\begin{proof} 
Introduce the orthogonal projector $\M{P} = \M{I}- \M{QQ}\t$.  
By the proof of~\cite[Theorem 4.3]{tropp2017practical},  \begin{equation}\label{eqn:twosided} \begin{aligned} \|\M{M}-  \M{Q}(\M{\Psi}\t\M{Q})^\dagger\M{\Psi}\t\M{M}\|_F^2 = & \>  \|\M{PM}\|_F^2  + \| (\M{\Psi}\t\M{Q})^\dagger \M\Psi\t \M{PM}\|_F^2 \\
\le &\> \|\M{PM}\|_F^2  + \|(\M{\Psi}\t\M{Q})^\dagger\|_2^2 \|  \M{\Psi}\t\M{PM}\|_F^2,   \end{aligned}\end{equation}
where in the last step, we have used submultiplicativity.

Let $\mc{E}$ denote the event 
\[ \|\M{M}-\M{QQ}\t\M{M}\|_F^2 \le (1 + 2(1+\Gamma_r))\|\M\Sigma_\perp\|_F^2.\] 
The assumptions on $\ell_r$ ensure that, by \cref{thm:rrf}, $\prob\{\mc{E}^c\}\le \delta/2$. Let us condition on the event $\mc{E}$. Noting that $\M\Omega$ and $\M\Psi$ are independent, as in the proof of Theorem~\ref{thm:rrf}, 
\[ \|(\M{\Psi}\t\M{Q})^\dagger\|_2^2 \|  \M{\Psi}\t\M{PM}\|_F^2 \le 2 (1+\Gamma_l)\|\M{PM}\|_F^2 \le 2(1+\Gamma_l)(1+2(1+\Gamma_r))\|\M\Sigma_\perp\|_F^2,\] 
holds with probability at least $1-\delta/2$. We can remove the conditioning to obtain with probability at least  $1-\delta$, the two inequalities hold simultaneously
\begin{align*}
    \|\M{PM}\|_F^2 \le & \>  (1 + 2(1+\Gamma_r))\|\M\Sigma_\perp\|_F^2\\
    \|(\M\Psi\t\M{Q})^\dagger \M{PM}\|_2^2 \le & \> 2\left(1 + 2\Gamma_l\right) (1 + 2(1+\Gamma_r))\|\M\Sigma_\perp\|_F^2.
\end{align*}
Plug these into~\eqref{eqn:twosided} to complete the proof.
\end{proof}

\section{Application: Tucker compression with Khatri-Rao Products}\label{sec:tucker}
In this section, we present an application of \krp{} projections to tensor compression in the Tucker format. \changes{Structured sketching for Tucker decomposition has been considered previously. For example, \cite{minster2024parallel} uses Kronecker-structured matrices, while \cite{CW19,che2025efficient,sun2020low} employ KRP-structured random matrices for randomized Tucker approximation.} 

Kressner and Peri\v{s}a \cite{kressner2017recompression} propose a randomized HOSVD algorithm for computing a Tucker approximation of tensors that are Hadamard (elementwise) products of two tensors in Tucker format.
Kressner and Peri\v{s}a present numerical experiments that show that the accuracy of recompression using random sketches based on Kronecker products of vectors is at least as accurate as feasible alternatives \cite[Section 7]{kressner2017recompression}.
They also point out that ``The analysis of this randomized algorithm is subject to further work.''
The result presented in  Theorem~\ref{thm:rhsovd_krp} fills this gap.

We discuss KRP projection–based Tucker decomposition algorithms in \Cref{ssec:tucker_algorithms}, including a new variant called RHOSVD-KRP-MEMO. We derive new results on the approximation properties of the resulting algorithms in \Cref{ssec:tucker_analysis}. We conclude the section with some numerical illustrations in \Cref{ssec:tucker_numerical}.

\subsection{Algorithms}\label{ssec:tucker_algorithms}
We can employ \changes{\krp{}} in randomized algorithms for Tucker decompositions. Consider random projection based variants of HOSVD and STHOSVD methods introduced in \cite{zhou2014decomposition,MSK20} and \cite{randtucker_survey,MSK20}, respectively. The main difference is that the sketching matrix $\M\Omega$ is taken to be a \krp{}, instead of dense Gaussian matrices. This has the advantage that generating the sketching matrices requires less work. The computational step changes from matrix multiplication in the case of Gaussians to matricized tensor times Khatri-Rao products (MTTKRP) in the case of \krp{} matrices. 

In a bit more detail, for HOSVD, consider a tensor $\T{X} \in \R^{n_1 \times \dots \times n_d}$, target rank $\V{r} = (r_1,\dots,r_d)$, and oversampling parameter $p$ with $\ell_i = r_i+p$ for all modes $1 \le i \le d$. We compute the randomized range finder of each mode unfolding $\M{X}_{(i)}$ by computing the MTTKRP $\M{Y}_{(i)} \leftarrow \M{X}_{(i)}( \M{\Omega}^{(d)} \odot \dots \odot \M{\Omega}^{(i+1)} \odot \M{\Omega}^{(i-1)} \odot \dots \odot \M{\Omega}^{(1)})$ for appropriately-sized random matrices $\{\M{\Omega}^{(i)}\}_{i=1}^d$.  The factor matrices are computed by computing the thin-QR factorization $\M{Y}_{(i)} = \M{Q}_i\M{R}_i$ for $1\le i \le d$.  Once all factor matrices $\{\M{Q}_i\}_{i=1}^d$ are computed, we obtain the core tensor via the multi-TTM operation $\T{G} = \T{X} \times_1 \M{Q}_1^\top \times_2 \M{Q}_2^\top \times \dots \times_d \M{Q}_d^\top$ at the end. \changes{It is worth mentioning that each time we form the sketch $\M{Y}_{(i)}$, we draw a fresh sequence of independent random matrices $\{\M{\Omega}^{(i)}\}_{i=1}^d$.} A variation to this is discussed shortly, which we refer to as memoization, only generates a single sequence of random matrices. While memoization has been applied in other contexts involving multiple MTTKRPs \cite[Section 3.6.2]{ballard2025tensor}, to the best of our knowledge, its use in randomized Tucker approximation has not been previously explored.  \changes{The work \cite{sun2020low} also employs a single sequence of random matrices; however, it uses this design primarily to reduce storage requirements and does not exploit memoization to reuse intermediate computations.}

\begin{algorithm}[!ht]
     \caption{Randomized HOSVD with Khatri-Rao products}
    \begin{algorithmic}[1]
    \REQUIRE Tensor $\T{X}$, target rank $\V{r} = (r_1,r_2,\dots,r_d)$, oversampling parameter $p$
    \FOR{$i=1,\dots,d$}
        \STATE Draw \changes{a new sequence of} $d-1$ random matrices $\M{\Omega}^{(j)} \in \R^{n_j \times \ell_i}$ for $j \neq i$     \STATE\label{line:rhosvd:mttkrp}Compute MTTKRP $\M{Y}_{(i)} \leftarrow \M{X}_{(i)}( \M{\Omega}^{(d)} \odot \dots \odot \M{\Omega}^{(i+1)} \odot \M{\Omega}^{(i-1)} \odot \dots \odot \M{\Omega}^{(1)})$
        \STATE\label{line:rhosvd:qr} Compute thin QR $\M{Y}_{(i)} = \M{Q}_i \M{R}$
    \ENDFOR
    \STATE\label{line:rhosvd:core} Compute core $\T{G} = \T{X} \times_1 \M{Q}_1^\top \times \cdots \times_d \M{Q}_d^\top$
    \ENSURE $\hat{\T{X}} = [\T{G}; \{\M{Q}_i\}_{i=1}^d]$    
    \end{algorithmic}
    \label{alg:rhosvd_krp}
\end{algorithm}

In the case of STHOSVD, we initialize the core tensor $\T{G}^{(0)} = \T{X}$.  At step $i$, the sketch is computed using the current core tensor using MTTKRPs, i.e., 
$$\M{Y}_{(i)} \leftarrow \M{G}_{(i)}^{(i-1)}  \left( \M{\Omega}^{(d)} \odot \dots \odot \M{\Omega}^{(i+1)} \odot \M{\Omega}^{(i-1)} \odot \dots \odot \M{\Omega}^{(1)} \right).$$ We then compute a thin QR factorization $\M{Y}_{(i)} = \M{Q}_i \M{R}_i$, and the matrix $\M{Q}_i \in \R^{n_i \times \ell_i}$ forms the orthonormal basis (factor matrix) for mode $i$. The core for the subsequent steps are updated as $\T{G}^{(i)} = \T{G} \times_i \M{Q}_i\t$. In describing STHOSVD, we have assumed that the modes are processed in order $\{1,\dots,d\}$. However, the algorithm can be readily modified to incorporate a different processing order. As in the HOSVD case, we draw a fresh sequence of independent random matrices each time we form the sketch.

The details for the HOSVD and STHOSVD algorithms employing Khatri-Rao products of random matrices as described are shown in \cref{alg:rhosvd_krp,alg:rst_krp}, respectively. \changes{Adaptive versions of these algorithms are proposed in~\cite{CW19,che2025efficient}. Our contributions here are improved analysis of the existing results.}

\begin{algorithm}[!ht]
    \begin{algorithmic}[1]
    \REQUIRE Tensor $\T{X}$, target rank $\V{r} = (r_1,r_2,\dots,r_d)$, oversampling parameter $p$
    \STATE Set $\T{G} = \T{X}$
    \FOR{$i=1,\dots,d$}
        \STATE Draw $d-1$ random matrices $\M{\Omega}^{(j)} \in \R^{\ell_j \times \ell_i}$ for $j < i$, $\M{\Omega}^{(j)} \in \R^{n_j \times \ell_i}$ for $j > i$ 
        \STATE\label{line:rst:mttkrp} Compute MTTKRP $\M{Y}_{(i)} \leftarrow \M{G}_{(i)}( \M{\Omega}^{(d)} \odot \dots \odot \M{\Omega}^{(i+1)} \odot \M{\Omega}^{(i-1)} \odot \dots \odot \M{\Omega}^{(1)})$
        \STATE\label{line:rst:qr} Compute thin QR $\M{Y}_{(i)} = \M{Q}_i \M{R}$
        \STATE\label{line:rst:ttm} Update core $\T{G} \leftarrow \T{G} \times_i \M{Q}_i^\top$
    \ENDFOR
    \ENSURE $\hat{\T{X}} = [\T{G}; \{\M{Q}_i\}_{i=1}^d]$    
    \end{algorithmic}
    \caption{Randomized STHOSVD with Khatri-Rao products}
    \label{alg:rst_krp}
\end{algorithm}

\paragraph{Computational Cost and Benefits}
To simplify the analysis of the computational complexity, we assume a $d$-mode tensor $\T{X} \in \R^{n \times \dots \times n}$ with uniform mode size and target rank $\V{r} = (r, \dots, r)$, and let $\ell = r + p$ for oversampling parameter $p$. At each mode, both \cref{alg:rhosvd_krp,alg:rst_krp} perform thin-QR factorizations of $n \times \ell$ matrices with cost $4n\ell^2$.  In \cref{alg:rst_krp}, the MTTKRP (Line~\ref{line:rst:mttkrp}) and TTM (Line~\ref{line:rst:ttm}) at each mode operate on the progressively compressed core tensor $\T{G}$, each with cost $2\ell^i n^{d - i + 1}$ flops, resulting in a total cost of
$$4dn\ell^2 + \sum_{i=1}^d 4\ell^i n^{d - i + 1}\> \text{flops}.$$

In contrast, Algorithm~\ref{alg:rhosvd_krp} computes each MTTKRP (Line~\ref{line:rhosvd:mttkrp}) on the original tensor $\T{X}$, each costing $2n^d \ell$, and performs a multi-TTM on the original tensor to form the core (Line~\ref{line:rhosvd:core}), with cost $\sum_{i=1}^d 2\ell^i n^{d-i+1}$ flops, yielding a total cost of
$$
4dn\ell^2 + 2dn^d \ell  + \sum_{i=1}^d 2\ell^i n^{d-i+1} \> \text{flops}.
$$

In Algorithm~\ref{alg:rhosvd_krp},  we generate an independent set of random matrices for each mode unfolding. In a process known as memoization, we can instead generate a single set of random matrices $\{\M{\Omega}_{i}\}_{i=1}^d$ and use the same random matrices $d$ times to compute the sketches.  We can reduce the overall cost of the computation of the sequence of multi-MTTKRPs by a factor of $d/2$ using memoization. For details on the use of memoization to compute a sequence of MTTKRPs efficiently, see~\cite[Section 3.6.2]{ballard2025tensor}. With memoization, the total cost becomes
$$
4dn\ell^2 + 4n^d \ell  + \sum_{i=1}^d 2\ell^i n^{d-i+1} \> \text{flops}.
$$

We denote Algorithm~\ref{alg:rhosvd_krp} as RHOSVD-KRP and its memoized version by RHOSVD-KRP-MEMO,  Algorithm~\ref{alg:rst_krp} as RSTHOSVD-KRP, and refer to their randomized variants that use dense Gaussian random matrices as RHOSVD and RSTHOSVD, respectively. We compare the leading-order computational costs of these algorithms, along with the deterministic HOSVD and STHOSVD, in Table~\ref{tab:cost}. For lower approximation accuracy (on the order of the square root of machine precision), the leading-order cost of HOSVD and STHOSVD can be reduced to $dn^{d+1}$ and $n^{d+1}$, respectively, by using the Gram matrix approach to compute the leading left singular vectors.

\begin{table}[!ht]
    \centering
    \begin{tabular}{l|cccc}
        \hline
        \multicolumn{5}{c}{HOSVD variants} \\ \hline
        Method & HOSVD & RHOSVD & RHOSVD-KRP & MEMO \\
        \hline
        Cost & $2dn^{d+1}$ & $2d n^d \ell $ & $2d n^d \ell $ & $6n^d\ell$ \\ 
        RNGs &  - & $\ell d n^{d-1}$ & $d(d-1)n\ell$ & $(d-1)n\ell$   \\ \hline
        \multicolumn{5}{c}{STHOSVD variants} \\ \hline
        Method & STHOSVD & RSTHOSVD & RSTHOSVD-KRP & -\\ \hline
        Cost& $2n^{d+1}$ & $4n^d\ell $ & $4n^d \ell $ &  -  \\
        RNGs &  - &  $\ell n^{d-1}$ & $\frac{d(d-1)}{2}n\ell$ & - 
    \end{tabular}
    \caption{Leading-order computational cost in flops and number of random numbers generated (RNGs) for various HOSVD and STHOSVD variants. Here, MEMO refers to RHOSVD-KRP-MEMO.} \label{tab:cost}   
\end{table}

One substantial benefit of using Khatri-Rao products of random matrices is the reduction in random entries we need to generate. 
While processing each mode in RHOSVD, drawing a dense random matrix requires $\ell n^{d-1}$ random entries, while drawing the random matrices for Algorithm~\ref{alg:rhosvd_krp} requires only $(d-1)n\ell$ random entries. In case of RSTHOVD, while processing the $i$-th mode, drawing a dense random matrix requires $\ell^{i} n^{d-i}$ entries, while drawing the random matrices for Algorithm~\ref{alg:rst_krp} requires only $(d-i)n\ell+(i-1)\ell^2$ random entries.
The performance benefits of this reduction are clearly shown in our numerical experiments in \Cref{ssec:tucker_numerical} and are summarized in \cref{tab:cost}.

\subsection{Analysis}\label{ssec:tucker_analysis}
We extend the analysis of RHOSVD and RSTHOSVD presented in \cite{MSK20} to the Khatri–Rao product based variants, \cref{alg:rhosvd_krp,alg:rst_krp}.

\begin{theorem}[RHOSVD-KRP] \label{thm:rhsovd_krp}
Let $\T[\hat]{X}$ be the approximation to $\T{X} \in \R^{n_1 \times \dots \times n_d}$ generated by \cref{alg:rhosvd_krp} or its memoized version RHOSVD-KRP-MEMO with target rank $(r_1,\dots,r_d)$. Let $0< \delta \leq \min\{1,2de^{-2}\}$. If $ \ell_i \geq   C_i^d \ln^d(2d\cdot 9^{r_i}/\delta)$ for $1 \le i \le d$, then with with probability at least $1-\delta$
    $$\| \T{X} - \T[\hat]{X} \|_F^2 \leq \sum_{i=1}^d  (1 + 2 (1+\Gamma_i))\sum_{j=r_i+1}^{n_i} \sigma_j^2 (\M{X}_{(i)}) $$ 
    where 
    \[ \Gamma_i \equiv   (C_i')^d\max\left\{\frac{\ln^{1/2}(2d/\delta)}{\ell_i^{1/2}},  \frac{\ln^d(2d/\delta)}{\ell_i} \right\}  \quad \text{ for } \quad 1 \le i \le d. \]
Here, $C_i$ and $C_i'$ are constants that only depend on $K$ for $1 \le i \le d$.
\end{theorem} 
\begin{proof}
We first give the proof for the output of \cref{alg:rhosvd_krp}. 

Let $\hat{\T{X}} = [\T{G}; \{\M{Q}_i\}_{i=1}^d]$ be the approximation of $\T{X}$ obtained by \cref{alg:rhosvd_krp}. For $i = 1$ to $d$,  the matrices $\M{Q}_i \M{Q}_i^\top$ are orthogonal projectors. Hence, from \cite[Theorem 5.1]{vann2012new}, we can express the approximation error as
\begin{equation*}
\begin{aligned}
    \| \T{X} - \hat{\T{X}} \|_F^2 &= \left\| \T{X} - \T{X} \times_1 \M{Q}_1\M{Q}_1^\top \times_2 \dots \times_d \M{Q}_d\M{Q}_d^\top \right\|_F^2 \\
    &\leq \sum_{i=1}^d \left\| \T{X} \times_i (\M{I} - \M{Q}_i \M{Q}_i^\top) \right\|_F^2 = \sum_{i=1}^d \left\| (\M{I} - \M{Q}_i \M{Q}_i^\top) \M{X}_{(i)} \right\|_F^2.
\end{aligned}
\end{equation*}
Define the events
\[
\mathcal{E}_i := \left\{ \left\| (\M{I} - \M{Q}_i \M{Q}_i^\top) \M{X}_{(i)} \right\|_F^2 \leq (1 + 2 (1+\Gamma_i)) \sum_{j = r_i + 1}^{n_i} \sigma_j^2(\M{X}_{(i)}) \right\} \qquad 1 \le i \le d.
\]
Then, by \cref{thm:rrf}, for $\ell_i \geq C_i^d \, \ln^d(2d\cdot 9^{r_i}/\delta)$,  $\prob(\mathcal{E}_i) \geq 1 - \delta/d$ and $\prob(\mathcal{E}_i^c) \leq  \delta/d$. Applying de Morgan's law and  the union bound, we get
$$
\prob\left( \bigcap_{i=1}^d \mathcal{E}_i \right) = 1 - \prob\left( \bigcup_{i=1}^d \mathcal{E}_i^c \right) \geq 1 - \sum_{i=1}^d \prob(\mathcal{E}_i^c) \geq 1 - \delta.
$$
Therefore, with probability at least $1 - \delta$, all mode-wise error bounds hold simultaneously, and we have
\[
\left\| \T{X} - \hat{\T{X}} \right\|_F^2 \leq \sum_{i=1}^d (1 + 2 (1+\Gamma_i)) \sum_{j = r_i + 1}^{n_i} \sigma_j^2(\M{X}_{(i)}).
\]
This completes the proof for the output of \cref{alg:rhosvd_krp}. The proof of the memoized version is essentially the same. In the case of  \cref{alg:rhosvd_krp}, the events $\{\mc{E}_i\}_{i=1}^d$ are independent, but they are not in the memoized version. The proof holds for the memoized version, since the union bound holds regardless of whether the events are independent.  
\end{proof}
We extend this analysis to RSTHOSVD-KRP. 
\begin{theorem}[RSTHOSVD-KRP]
Let  $\T{X} \in \R^{n_1 \times \dots \times n_d}$ be an input to \cref{alg:rst_krp} with target rank $(r_1,\dots,r_d)$  and oversampling parameter $p$ such that $\ell_i = r_i+p$ for $1 \leq i \leq d$. Let $\T[\hat]{X}$ be the approximation to $\T{X}$. If $0 < \delta \leq \min\{1,2de^{-2}\}$ and $\ell_i \geq  C_i^d \ln^d(2d\cdot 9^r/\delta)$ for $1 \le i \le d$, then with probability at least $1-\delta$ 
    $$\| \T{X} - \T[\hat]{X} \|_F^2 \leq \sum_{i=1}^d  (1 + 2 (1+\Gamma_i)) \sum_{j=r_i+1}^{n_i} \sigma_j^2 (\M{X}_{(i)}) $$ 
    where 
    \[ \Gamma_i \equiv   (C_i')^d\max\left\{\frac{\ln^{1/2}(2d/\delta)}{\ell_i^{1/2}},  \frac{\ln^d(2d/\delta)}{\ell_i} \right\}  \quad \text{ for }   \quad 1 \le i \le d. \]
Here, $C_i$ and $C_i'$ are constants that only depend on $K$.
\end{theorem}

\begin{proof}
Let $\T{G}^{(i)} = \T{X} \times_1 \M{Q}_1^\top \times_2 \dots \times_i \M{Q}_i^\top$ be the core tensor that is truncated in the first $i$ modes for $1 \le i \le d$ with $\T{G}^{(0)} = \T{X}$. Furthermore, let $\hat{\T{X}}^{(i)} = \T{G}^{(i)} \times_1 \M{Q}_1 \times_2 \dots \times_i \M{Q}_i$ be the resulting partial approximation for $1 \le i \le d$, with $\T{X}^{(0)} = \T{X}$. From \cite[Theorem 5.1]{vann2012new}, we can express the error in $\T[\hat]{X}$ as 
\begin{equation*}
    \begin{aligned}
        \| \T{X} - \hat{\T{X}} \|_F^2 &= \| \T{X} - \T{X} \times_1 \M{Q}_1\M{Q}_1^\top \times_1 \dots \times_d \M{Q}_d\M{Q}_d^\top \|_F^2 \\
        &= \sum_{i=1}^d \| \hat{\T{X}}^{(i-1)} - \hat{\T{X}}^{(i)} \|_F^2 \\
        &= \sum_{i=1}^d \| \T{G}^{(i-1)} \times_1 \M{Q}_1 \times_2 \dots \times_{i-1} \M{Q}_{\changes{i-1}} \times_i (\M{I}-\M{Q}_i\M{Q}_i^\top ) \|_F^2.
    \end{aligned}
\end{equation*}
Then, unfolding each term in the sum along mode $i$, we can express this as
\begin{equation*}
    \begin{aligned}
        \| \T{X} - \hat{\T{X}} \|_F^2 &= \sum_{i=1}^d \| (\M{I}- \M{Q}_i\M{Q}_i^\top ) \M{G}_{(i)}^{(i-1)} (\underbrace{\M{I} \kron \dots \kron \M{I}}_{d-i} \kron \M{Q}_{i-1} \kron \dots \kron \M{Q}_1) \|_F^2 \\
       &\leq \sum_{i=1}^d \| (\M{I}- \M{Q}_i\M{Q}_i^\top ) \M{G}_{(i)}^{(i-1)} \|_F^2,
    \end{aligned}
\end{equation*}
where the inequality comes as each $ \M{Q}_i$ has orthonormal columns, so the Kronecker product 
$ \M{I} \kron \dots \kron \M{I} \kron \M{Q}_{i-1} \kron \dots \kron \M{Q}_1 $
 does as well. 
 
 Define the events for $1 \le t \le d$, 
\[ \mc{E}_t = \left\{ \sum_{i=1}^t\| (\M{I}- \M{Q}_i\M{Q}_i^\top ) \M{G}_{(i)}^{(i-1)} \|_F^2 \le \sum_{i=1}^t(1 + 2 (1+\Gamma_i)) \sum_{j=r_i+1}^{n_i} \sigma_j^2 (\M{X}_{(i)})\right\}. \]
We know by \cref{thm:rrf}, $\prob\{ \mc{E}_1 \} \ge 1 -\delta/d$ and note that $ \mc{E}_1 \subseteq \dots \subseteq \mc{E}_t \subseteq \mc{E}_{t+1} \subseteq \dots \subseteq \mc{E}_d$. Conditional on the event $\mc{E}_t$, we can write for $1 \leq t \le d-1$
\[  
\begin{aligned}
\| (\M{I}- \M{Q}_{t+1}\M{Q}_{t+1}^\top ) \M{G}_{(t+1)}^{(t)} \|_F^2  \le& \>  (1 + 2 (1+\Gamma_{t+1})) \sum_{j=r_{t+1}+1}^{n_t} \sigma_j^2 (\M{G}^{(t)}_{(t+1)}),  \\ 
 \le & \>  (1 + 2 (1+\Gamma_{t+1})) \sum_{j=r_{t+1}+1}^{n_t} \sigma_j^2 (\M{X}_{(t+1)}).
\end{aligned}
\]
with probability at least $1-\delta/d$ due to \cref{thm:rrf}. In the last step, we have used the fact that the singular values of each partially truncated core tensor cannot be larger than the singular values of the original tensor $\T{X}$ as $\M{Q}_t$ has orthonormal columns for every $1 \le t\le d$ (see~\changes{\cite[proof of Theorem 3.2]{MSK20}}). Thus, for $1\leq t \leq  d-1$, we have $\prob\{ \mc{E}_{t+1} | \mc{E}_t \} \geq 1 - \delta/d$. 
From the law of total probability, we have
\[
\begin{aligned}
\prob\{\mc{E}_d\} &= \prob\{\mc{E}_d | \mc{E}_{d-1}\} \prob\{\mc{E}_{d-1}\} + \prob\{\mc{E}_d | \mc{E}_{d-1}^c\} \prob\{\mc{E}_{d-1}^c\} \\
&\geq \prob\{\mc{E}_d | \mc{E}_{d-1}\} \prob\{\mc{E}_{d-1}\} \\
& \geq \left(1 - \delta/d \right)\prob\{\mc{E}_{d-1} | \mc{E}_{d-2}\}  \prob\{\mc{E}_{d-2}\} \\
& \geq \left(1 - \delta/d \right)^2  \prob\{\mc{E}_{d-2}\}.
\end{aligned}
\]
Solving the above iteratively by conditioning on the previous event, we get
$$
\prob\{\mc{E}_d\} \geq  \left(1 - \delta/d \right)^d \geq 1 - d \delta/d = 1 -\delta.
$$
This completes the proof.
\end{proof}
 
The analysis of RHOSVD-KRP and RSTHOSVD-KRP is also provided in~\cite{CW19,che2025efficient}. 
Since the results in~\cite{che2025efficient} are more recent, we focus on comparing our results against theirs. 
The results in~\cite[Theorem 5.14]{che2025efficient} for \krp{}s of Gaussians, however, show explicit dependence on the mode sizes $N = \prod_{i=1}^dn_i$.  
The origin of the explicit dependence on $N$ is likely due to a crude bound in Lemma 5.11 of~\cite{che2025efficient}. 

\subsection{Numerical Illustrations}\label{ssec:tucker_numerical}
We illustrate the accuracy and performance of \cref{alg:rhosvd_krp,alg:rst_krp} on synthetic tensors.

\paragraph{Setup} We consider a synthetic tensor, which we call the Cauchy tensor, with entries defined as
\[ x_{i_1,i_2,i_3,i_4} = \frac{1}{(i_1^{\alpha} + \dots + i_4^{\alpha})^{1/\alpha}}, \qquad   1 \le i_j \le n, 1 \le j \le 4.  \]
The value of $\alpha$ controls the sharpness of the singular value decay; smaller $\alpha$ values result in sharper decay in the mode-wise singular values of $\T{X}$.

We verify the accuracy of RHOSVD-KRP (\cref{alg:rhosvd_krp}), its memoized version RHOSVD-KRP-MEMO, and RSTHOSVD-KRP (\cref{alg:rst_krp}) by comparing the relative error of these algorithms to that obtained by HOSVD and RHOSVD. We plot the relative error with increasing target rank for all the algorithms in  \cref{fig:cauchyp} for $\alpha=2$. We use the same target rank in each mode, and use no oversampling so that each algorithm computes an approximation of the same rank. 
From \cref{fig:cauchyp}, we can see that RHOSVD-KRP and RHOSVD-KRP-MEMO perform similarly to RHOSVD, and the performance of these three is comparable to HOSVD. 
A similar trend is observed when comparing RSTHOSVD-KRP with RSTHOSVD and STHOSVD. We observed a similar comparison among the proposed and baseline algorithms for Cauchy tensors generated using different values of $\alpha$.

\begin{figure}[!ht]
    \centering
    \includegraphics[width = 0.6\textwidth]{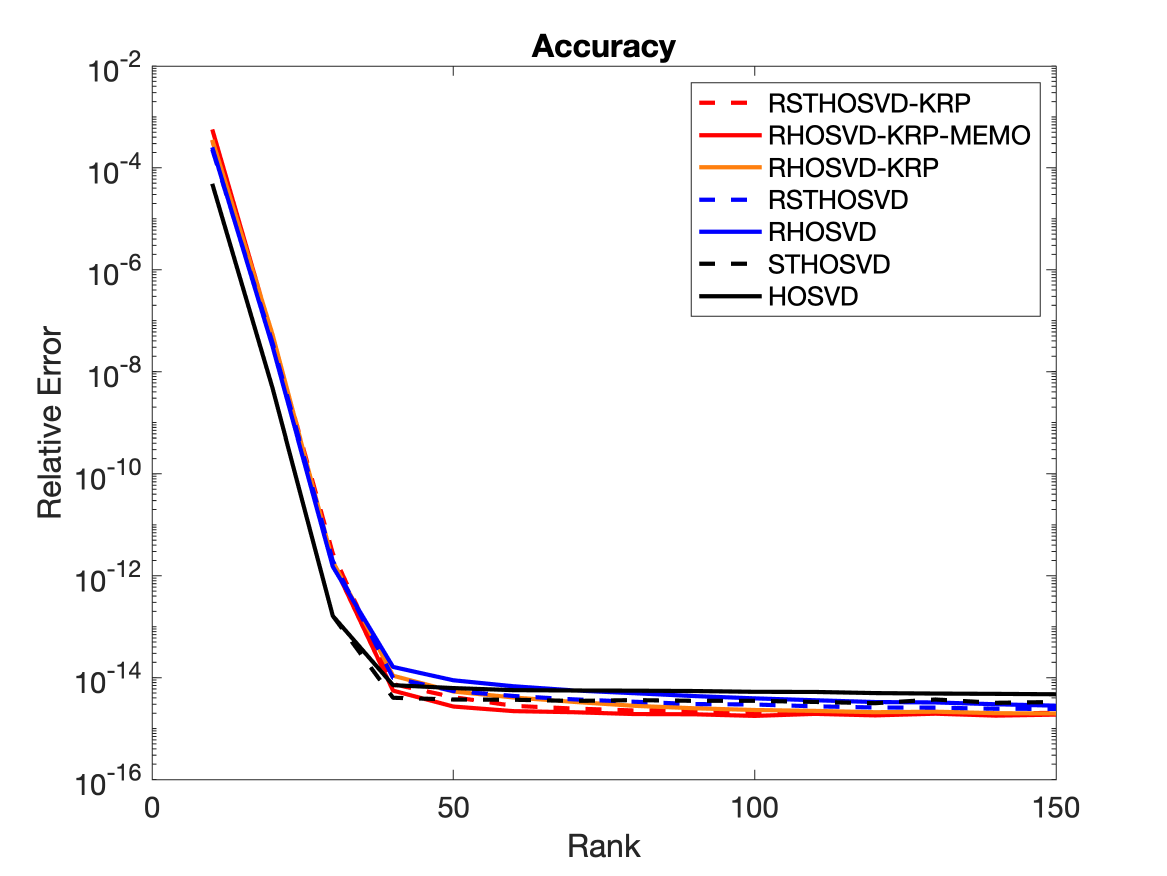}
    \caption{Plots of relative error versus target ranks for a 4-way Cauchy tensor with $n=250$ and exponents $\alpha = 2$.}
    \label{fig:cauchyp}
\end{figure}

For $\alpha=2$,  we also compare the runtime of RHOSVD-KRP, its memoized variant RHOSVD-KRP-MEMO, and RSTHOSVD-KRP with baseline algorithms. In \cref{fig:syn4d}, we present the average total runtime (left panel) of all algorithms as the target rank $r$ increases, along with a detailed time breakdown for the randomized methods (right panel). 

If the target rank is low, both RSTHOSVD-KRP and RSTHOSVD have similar runtimes and are substantially faster than STHOSVD. However, as the target rank increases, RSTHOSVD-KRP becomes significantly faster than RSTHOSVD.
The plot showing the breakdown reveals that this performance gap is primarily due to the high cost of generating the dense Gaussian random matrix in RSTHOSVD. In contrast, RSTHOSVD-KRP avoids this cost, spending negligible time on random number generation. 

\begin{figure}[!ht]
    \centering
    \includegraphics[width=0.98\textwidth]{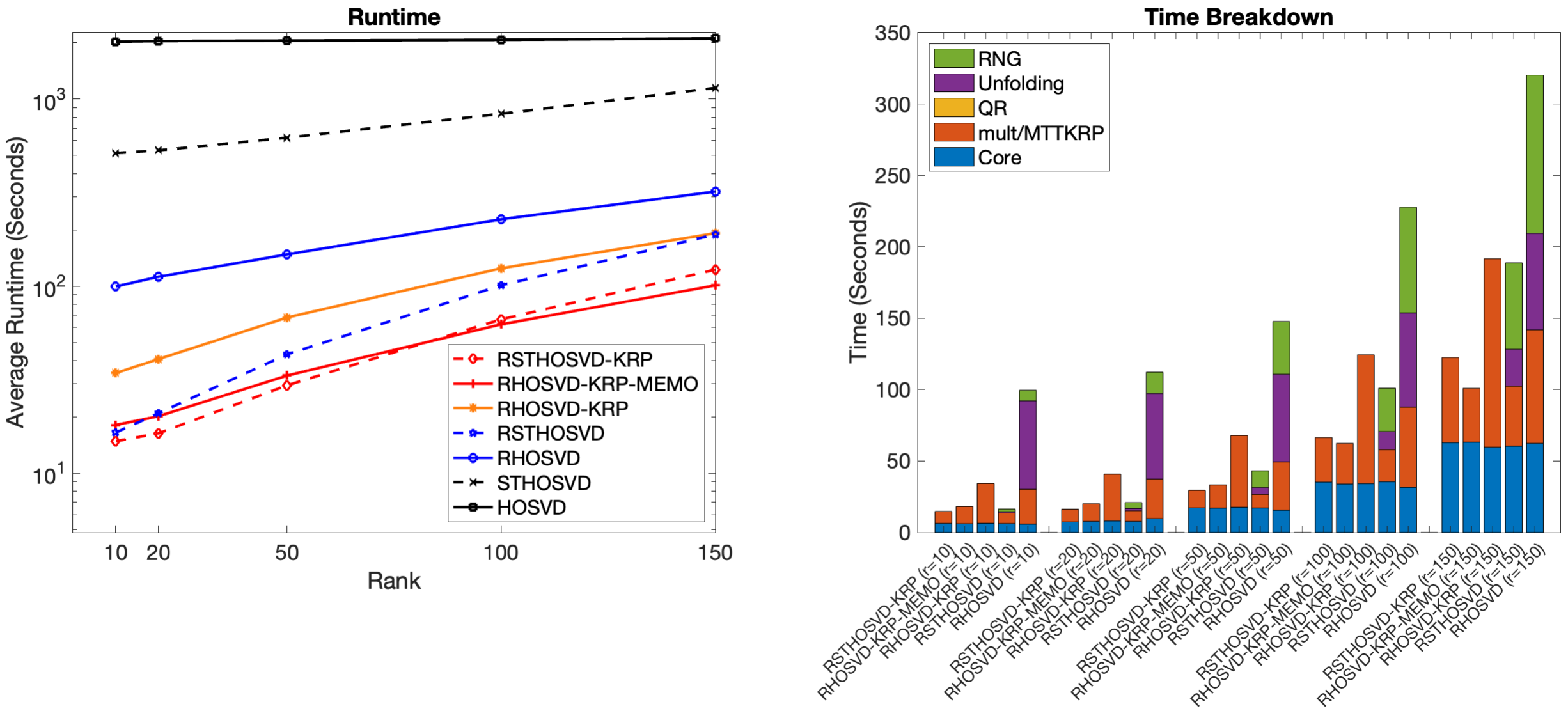}
    \caption{Running time in seconds of the different methods (left panel) on the Cauchy-like tensor averaged over three runs and a breakdown of the running times in different components (right panel). RNG refers to the time of generating pseudorandom numbers, `Unfolding' refers to the time required for matricization, `QR' refers to the time for orthogonalizing the sketches, `mult/KRP' refers to the time for matrix multiplication in the case of Gaussians and MTTKRP for \krp matrices, and `Core' refers to the time for forming the core tensor. }
    \label{fig:syn4d}
\end{figure}

 Now we discuss the HOSVD variants. RHOSVD-KRP consistently outperforms RHOSVD across all values of the target rank. For smaller target ranks, the tensor unfolding time in RHOSVD dominates the overall runtime, making it slower than RHOSVD-KRP, as shown in the breakdown plots. For higher target ranks, the performance gap between these two methods is primarily due to the increasing cost of generating the dense Gaussian matrix in RHOSVD, an overhead that is negligible in RHOSVD-KRP. The memoized variant, RHOSVD-KRP-MEMO, further accelerates RHOSVD-KRP by reducing the MTTKRPs computation time,  making it the fastest method, even outperforming RSTHOSVD-KRP at higher rank values.

\section{Application: Sensor Placement}\label{sec:flow}
In fluid dynamics, there is a pressing need to study high-resolution flow fields using experimental data. 
However, measurements of the flow fields are collected at only coarse grid point locations, thereby limiting the amount of information that can be extracted from these measurements. 
It is, therefore, crucially important to determine optimal locations for the placement of sensors to extract as much information from these limited measurements as possible. 

In this application, we are given a set of snapshots that represent the flow fields under different flow conditions that are split into a training set and a testing set. Let $d$ be the spatial dimension, so that the snapshot tensor can be expressed as $\T{S} \in \R^{N_1\times \dots \times N_d \times T}$, where $N_i$ for $1 \leq i \leq d$ represent the number of grid points in each spatial dimension and $T$ represents the number of training snapshots. The goal is to identify the locations (or indices) of the grid at which to collect the data and then recover the flow field given the measurements from a new flow field. 
We follow the tensor-based approach in~\cite{farazmand2023tensor}, and we refer the reader to the references within that paper for alternative approaches. 
There are two phases in this algorithm: training and test phases. 

\paragraph{Training phase} We first compute a Tucker low-rank approximation to the training tensor $\T{S}$ as 
\[ \T{S} \approx \T{G} \times_1 \M{Q}_1\M{Q}_1\t \times_2\dots \times_d \M{Q}_d\M{Q}_d\t,\]
where $\M{Q}_i\in \R^{N_i\times \ell_i}$ for $1\leq i \leq d$ have orthonormal columns. 
This can be computed using any of the algorithms described in \cref{sec:tucker}.  
Next, for each $1 \leq i \leq d$ we compute a column-pivoted QR factorization 
\[ \M{Q}_i\t \bmat{\M\Pi_1^{(i)} &\M\Pi_2^{(i)} }=  \M{Z}_1^{(i)} \bmat{\M{R}_{11}^{(i)} & \M{R}_{12}^{(i)}}. \]
where $\M{Z}_1^{(i)} \in \R^{\ell_i \times \ell_i}$ is orthogonal and $\M{R}_{11}^{(i)} \in \R^{\ell_i\times \ell_i}$ is nonsingular and upper triangular. The matrix $\bmat{\M\Pi_1^{(i)} &\M\Pi_2^{(i)} }$ is a permutation matrix. We extract the matrix $\M{P}_i \equiv \M\Pi_1^{(i)} \in \R^{N_i\times \ell_i}$ which contains columns from the identity matrix; next, we define the factor matrices $\M{A}_i \equiv \M{Q}_i (\M{P}_i\t\M{Q}_i)^{-1}$ for $1\leq i \leq d$. The matrix \changes{$\M{P}_i$} determines the index set $\mathcal{I}_i$ with cardinality $|\mathcal{I}_i| = \ell_i$. The optimal locations (in terms of indices) correspond to the Cartesian grid $\mc{I} = \mc{I}_1 \times \dots \times \mc{I}_d$ totaling $\prod_{i=1}^d \ell_i$ sensor locations. 

\paragraph{Test phase} Next, we show how to recover a flow field corresponding to a new tensor $\T{T} \in \R^{N_1\times \dots \times N_d}$. We measure/access the tensor only at a select few locations $\widehat{\T{T}}\equiv \T{T}(\mc{I}_1,\dots \mc{I}_d) = \T{T}\bigtimes_{i=1}^d\M{P}_i\t \in \R^{\ell_1 \times \dots \times \ell_d} $. An approximation to the tensor $\T{T}$ can be obtained as 
\[ \T{T} \approx \widehat{\T{T}} \times_1 \M{A}_1 \times_2 \dots \times_d \M{A}_d. \]

As discussed in~\cite[Section 2.3.2]{farazmand2023tensor}, the most expensive step is the computation of the factor matrices $\M{Q}_i$ for $1\leq i \leq d$, and the authors suggest the use of RHOSVD. 
We accelerate this step further using \cref{alg:rhosvd_krp}.  We demonstrate the performance of our approach on a flow data set with three spatial dimensions obtained from~\cite{Popinet04:Tangaroa}. The details of the dataset and the processing is given in~\cite[Section 3.3]{farazmand2023tensor}. Briefly, the number of grid points is $150\times 90 \times 60$ with $150$ snapshots in the training set and $61$ in the test set. We consider two experiments: first, we report the timings for the compression of the training tensor in the training phase; second, we report the accuracy of the testing phase. 
\begin{figure}[!ht]
    \centering    
    \includegraphics[width=0.9\textwidth]{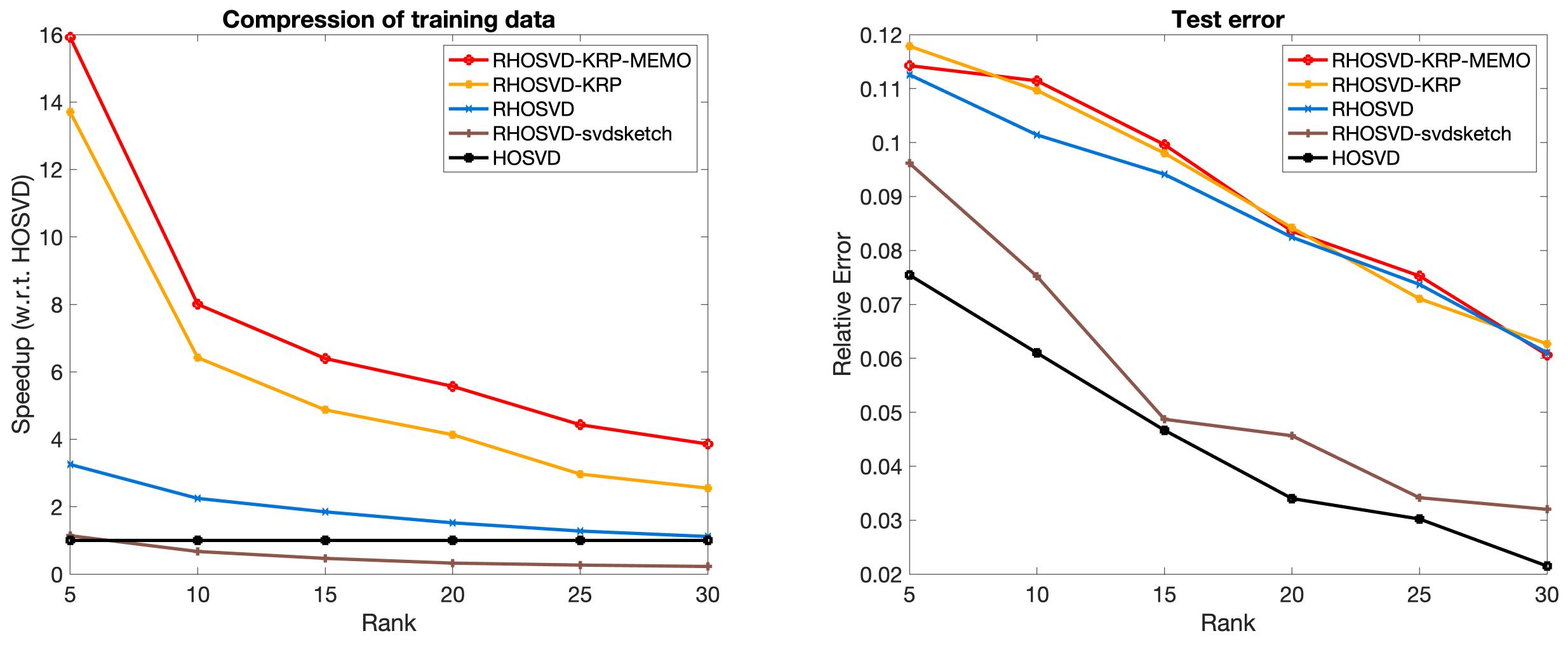}
    \caption{Left: Average speedup over $10$ runs for different rank values. Right: Average relative error over 10 runs for different rank values.}

    \label{fig:tang}
\end{figure}
\paragraph{Experiment 1: Compression of the training tensor} In this experiment, we compare the speedup of the following four methods w.r.t. HOSVD: RHOSVD-KRP (Algorithm~\ref{alg:rhosvd_krp}), its memoized version RHOSVD-KRP-MEMO, randomized HOSVD (RHOSVD), which uses Gaussian random matrices, and a variant of RHOSVD used in \cite{farazmand2023tensor}, which utilizes the MATLAB function \texttt{svdsketch} to compute the factor matrices. We refer to the latter implementation as RHOSVD-svdsketch. For HOSVD, we use the Gram approach to compute the factor matrices. We consider compression ranks of the form $(r,r,r,150)$, for  $r=5,10,15,20,25,30$; note that we do not compress in the last mode. The left panel of Figure~\ref{fig:tang} summarizes the average speedup over $10$ runs.

As we see from the plot, RHOSVD-KRP-MEMO is the fastest method, achieving a $16\times$ speedup compared to HOSVD at $r=5$. RHOSVD-KRP is the second fastest, with approx. $14\times$ speedup over HOSVD at the same rank. Among all methods, RHOSVD-svdsketch is the slowest one, even slower than the deterministic HOSVD. This is because, compared to random projection followed by the SVD approach used in RHOSVD, the \textit{svdsketch} function performs additional steps such as power iterations, error estimation, and adaptive rank selection, which improve its accuracy and robustness but introduce greater computational overhead.

\paragraph{Experiment 2: Accuracy of testing phase} We use a sensor grid of $r\times r \times r$ (similar to~\cite{farazmand2023tensor}) for $r=5,10,15,20,25,30$. We compute the flow field for each of the $61$ test flow fields and report the error averaged over all the snapshots. The right panel of \cref{fig:tang} summarizes the trade-off between relative error and the number of sensors for the test data for different HOSVD variants. The figure shows that the relative errors of RHOSVD-KRP-MEMO and RHOSVD-KRP are nearly identical and comparable to that of RHOSVD. The relative errors of HOSVD and RHOSVD-svdsketch are lower than those of RHOSVD-KRP-MEMO, RHOSVD-KRP, and RHOSVD, which is expected because HOSVD is a deterministic method and RHOSVD-svdsketch utilizes an additional power iteration step to improve accuracy. However, these methods are significantly slower than RHOSVD-KRP-MEMO and RHOSVD-KRP, making RHOSVD-KRP-MEMO and RHOSVD-KRP alternative choices to gain computational efficiency.

\section{Conclusions} \label{sec:conclusion}
This paper provides new analysis for \krp{}s and new algorithms for compressing block-structured matrices and tensors in the Tucker format. Compared to Gaussian random matrices, far fewer pseudo random numbers need to be generated, which can substantially reduce the run time. \changes{In our experiments, for low dimension $d$, the accuracy of the randomized algorithms with \krp{}s is comparable to Gaussian random matrices. } The analysis shows an exponential dependence on the dimension $d$, which appears to be extremely pessimistic, suggesting a gap still exists between theory and practice. A bevy of applications in matrix and tensor compressions shows the computational benefits of the \krp{}s.

\section*{Acknowledgements}
AKS is grateful to Elizaveta Rebrova for pointing us to \cite{haselby2023fast} and for helpful conversations. We also thank the reviewers for helpful comments and for pointing out an issue with an earlier draft. 

\appendix

\bibliographystyle{abbrv}
\bibliography{refs}
\end{document}